\documentclass[12pt]{amsart}
\pdfoutput=1

\usepackage[paperheight=11in, paperwidth=8.5in, left=1in, top=1.25in, right=1in, bottom=1.25in]{geometry}
\usepackage[linktocpage=true, linktoc=all, colorlinks=true, linkcolor=black, citecolor=black, filecolor=black, urlcolor=black, pagebackref=false, pdfstartpage={1}, pdfstartview={FitH}, pdftitle={}, pdfauthor={}, pdfsubject={}, pdfcreator={}, pdfproducer={}, pdfkeywords={}]{hyperref}
\usepackage{afterpage}
\usepackage{amsfonts}
\usepackage{amsmath}
\allowdisplaybreaks[4]
\usepackage{amssymb}
\usepackage{amsthm}
\usepackage{array}
\usepackage{arydshln}
\usepackage{bbm}
\usepackage[justification=centering]{caption}
\usepackage[capitalize,nameinlink,noabbrev,compress]{cleveref}

\usepackage{enumerate}
\usepackage{enumitem}
\usepackage{float}
\restylefloat{figure}
\usepackage{graphicx}
\usepackage{mathrsfs}
\usepackage[framemethod=tikz,ntheorem,xcolor]{mdframed}
\usepackage{multirow}
\usepackage{parcolumns}
\usepackage{scalerel}
\usepackage{tikz}\pdfpageattr{/Group <</S /Transparency /I true /CS /DeviceRGB>>}
\usetikzlibrary{arrows, calc, cd, decorations.markings, intersections, matrix, positioning, through}
\tikzset{commutative diagrams/.cd, every label/.append style = {font = \normalsize}}
\usepackage[all]{xy}
\usepackage[aligntableaux=center]{ytableau}

\numberwithin{equation}{section}
\newtheorem{thm}{Theorem}
\AfterEndEnvironment{thm}{\noindent\ignorespaces}
\numberwithin{thm}{section}

\newtheorem{conj}[thm]{Conjecture}
\AfterEndEnvironment{conj}{\noindent\ignorespaces}
\newtheorem{cor}[thm]{Corollary}
\AfterEndEnvironment{cor}{\noindent\ignorespaces}
\newtheorem{lem}[thm]{Lemma}
\AfterEndEnvironment{lem}{\noindent\ignorespaces}

\AfterEndEnvironment{prob}{\noindent\ignorespaces}
\newtheorem{prop}[thm]{Proposition}
\AfterEndEnvironment{prop}{\noindent\ignorespaces}

\theoremstyle{definition}
\newtheorem{defn}[thm]{Definition}
\AfterEndEnvironment{defn}{\noindent\ignorespaces}
\newtheorem{eg_no_qed}[thm]{Example}
\newenvironment{eg}[1][]{\begin{eg_no_qed}[#1]\pushQED{\qed}}{\popQED\end{eg_no_qed}}
\AfterEndEnvironment{eg}{\noindent\ignorespaces}
\newtheorem{rmk}[thm]{Remark}
\AfterEndEnvironment{rmk}{\noindent\ignorespaces}

\theoremstyle{remark}

\AfterEndEnvironment{claim}{\noindent\ignorespaces}
\newtheorem*{claimpf_no_qed}{Proof of Claim}

\AfterEndEnvironment{claimpf}{\noindent\ignorespaces}

\renewcommand{\eqref}[1]{\hyperref[#1]{\textup{(\ref*{#1})}}}

\DeclareMathOperator{\Fl}{Fl}

\DeclareMathOperator{\GL}{GL}

\DeclareMathOperator{\Gr}{Gr}

\DeclareMathOperator{\SL}{SL}

\DeclareMathOperator{\spn}{span}

\DeclareMathOperator{\Wr}{\mathsf{Wr}}

\newcommand{\detstar}{\operatorname{det}^*}
\newcommand{\psubset}{\mathcal{P}}
\DeclareMathOperator{\rev}{\mathsf{rev}}
\newcommand{\rnc}{\gamma}
\newcommand{\Rnc}[1]{\Gamma_{#1}}
\newcommand{\shift}[1]{\operatorname{\mathsf{S}}_{#1}}
\newcommand{\sumof}[1]{\sum\hspace*{-1pt}{#1}}
\newcommand{\vand}[1]{\mathsf{c}_{#1}}
\newcommand{\zeros}[1]{Z_{#1}}

\title{Wronskians, total positivity, and real Schubert calculus}

\author{Steven N. Karp}
\address{Department of Mathematics, University of Notre Dame}
\email{\href{mailto:skarp2@nd.edu}{skarp2@nd.edu}}

\thanks{The author was partially supported by an NSERC postdoctoral fellowship.}

\begin{document}

\begin{abstract}
A complete flag in $\mathbb{R}^n$ is a sequence of nested subspaces $V_1 \subset \cdots \subset V_{n-1}$ such that each $V_k$ has dimension $k$. It is called {\itshape totally nonnegative} if all its Pl\"{u}cker coordinates are nonnegative. We may view each $V_k$ as a subspace of polynomials in $\mathbb{R}[x]$ of degree at most $n-1$, by associating a vector $(a_1, \dots, a_n)$ in $\mathbb{R}^n$ to the polynomial $a_1 + a_2x + \cdots + a_nx^{n-1}$. We show that a complete flag is totally nonnegative if and only if each of its Wronskian polynomials $\Wr(V_k)$ is nonzero on the interval $(0, \infty)$. In the language of Chebyshev systems, this means that the flag forms a Markov system or $ECT$-system on $(0, \infty)$. This gives a new characterization and membership test for the totally nonnegative flag variety. Similarly, we show that a complete flag is totally positive if and only if each $\Wr(V_k)$ is nonzero on $[0, \infty]$. We use these results to show that a conjecture of Eremenko (2015) in real Schubert calculus is equivalent to the following conjecture: if $V$ is a finite-dimensional subspace of polynomials such that all complex zeros of $\Wr(V)$ lie in the interval $(-\infty, 0)$, then all Pl\"{u}cker coordinates of $V$ are real and positive. This conjecture is a totally positive strengthening of a result of Mukhin, Tarasov, and Varchenko (2009), and can be reformulated as saying that all complex solutions to a certain family of Schubert problems in the Grassmannian are real and totally positive. We also show that our conjecture is equivalent to a totally positive version of the secant conjecture of Sottile (2003).
\end{abstract}

\maketitle

\section{Introduction}\label{sec_introduction}

\noindent Let $\Fl_n(\mathbb{R})$ denote the {\itshape complete flag variety}, consisting of all sequences $V = (V_1 \subset \cdots \subset V_{n-1})$ of nested subspaces of $\mathbb{R}^n$ such that each $V_k$ has dimension $k$. Lusztig \cite{lusztig94} introduced two remarkable subsets of $\Fl_n(\mathbb{R})$, called the {\itshape totally nonnegative part} $\Fl_n^{\ge 0}$ and the {\itshape totally positive part $\Fl_n^{>0}$}. They may be defined as the subsets where all Pl\"{u}cker coordinates are nonnegative or positive, respectively, up to rescaling. The totally nonnegative parts of flag varieties have been widely studied, with connections to representation theory \cite{lusztig94}, combinatorics and cluster algebras \cite{fomin_williams_zelevinsky}, high-energy physics \cite{arkani-hamed_bai_lam17}, topology \cite{galashin_karp_lam22}, and many other topics.

On the other hand, the {\itshape Wronskian} of sufficiently differentiable functions $f_1(x), \dots, f_k(x)$ (defined on $\mathbb{R}$ or $\mathbb{C}$) is
\begin{align}\label{wronskian_formula}
\Wr(f_1, \dots, f_k) := \det\Big(\frac{d^{i-1}f_j}{dx^{i-1}}\Big)_{1 \le i,j \le k} = \det\begin{bmatrix}
f_1 & \cdots & f_k \\
f_1' & \cdots & f_k' \\
\vdots & \ddots & \vdots \\
f_1^{(k-1)} & \cdots & f_k^{(k-1)}
\end{bmatrix}.
\end{align}
The Wronskian is identically zero if $f_1, \dots, f_k$ are linearly dependent; otherwise, it only depends on the linear span $V$ of $f_1, \dots, f_k$ up to multiplication by a nonzero scalar. In particular, it makes sense to write $\Wr(V)$, and its zeros are well-defined.

The Wronskian appears in various contexts; we give three examples. First, when $k=2$, we have
$$
\Wr(f,g) = fg' - f'g = f^2\Big(\frac{g}{f}\Big)'.
$$
Hence if $f$ and $g$ are polynomials with no common factors, then the zeros of $\Wr(f,g)$ are the critical points of the rational function $\frac{g}{f}$. Second, when $f_1, \dots, f_k$ are linearly independent, the unique homogeneous linear differential operator $\mathcal{L}$ of order $k$ with leading coefficient $1$ and kernel spanned by $f_1, \dots, f_k$ is given by
\begin{align}\label{differential_operator}
\mathcal{L}(g) = \frac{\Wr(f_1, \dots, f_k, g)}{\Wr(f_1, \dots, f_k)} = \frac{d^kg}{dx^k} + \cdots.
\end{align}
Third, when $f_1, \dots, f_k$ are linearly independent and $r$ is a scalar, by interpolation there exists a nonzero $g$ in the linear span of $f_1, \dots, f_k$ with a zero of order at least $k-1$ at $r$. For a generic $r$, the zero of $g$ at $r$ has order exactly $k-1$; it is precisely when $r$ is a zero of $\Wr(f_1, \dots, f_k)$ that there exists a $g$ with a zero of order at least $k$ at $r$.

In this paper, we introduce a new connection between total positivity and Wronskians, and use it to show the hidden role that total positivity plays in certain conjectures in real Schubert calculus. We now explain our main results.

\subsection*{Complete flags and their Wronskians}
Let $\mathbb{R}[x]_{\le n-1}$ denote the subspace of $\mathbb{R}[x]$ of polynomials of degree at most $n-1$. We identify $\mathbb{R}^n$ with $\mathbb{R}[x]_{\le n-1}$, as follows:
\begin{align}\label{polynomial_identification}
\mathbb{R}^n \leftrightarrow \mathbb{R}[x]_{\le n-1}, \quad (a_1, \dots, a_n) \leftrightarrow a_1 + a_2x + \cdots + a_nx^{n-1}.
\end{align}
In particular, we may view a complete flag $(V_1, \dots, V_{n-1}) \in \Fl_n(\mathbb{R})$ as a sequence of nested subspaces inside $\mathbb{R}[x]_{\le n-1}$. Our first main result characterizes the totally nonnegative and totally positive flag varieties $\Fl_n^{\ge 0}$ and $\Fl_n^{>0}$ in terms of their Wronskian polynomials:
\begin{thm}\label{flag_wronskian}
Let $V = (V_1, \dots, V_{n-1}) \in \Fl_n(\mathbb{R})$.
\begin{enumerate}[label=(\roman*), leftmargin=*, itemsep=2pt]
\item\label{flag_wronskian_nonnegative} The flag $V$ is totally nonnegative if and only if $\Wr(V_k)$ is nonzero on the interval $(0,\infty)$, for all $1 \le k \le n-1$.
\item\label{flag_wronskian_positive} The flag $V$ is totally positive if and only if $\Wr(V_k)$ is nonzero on the interval $[0,\infty]$, for all $1 \le k \le n-1$. (Here, $\Wr(V_k)$ being nonzero at $\infty$ means that it has the expected degree $k(n-k)$.)
\end{enumerate}

\end{thm}

In the language of Chebyshev systems, the conclusions above say that $V$ forms a {\itshape Markov system} (or {\itshape $ECT$-system}) on $(0,\infty)$ and $[0,\infty]$, respectively. An equivalent characterization is that for all $1 \le k \le n-1$, every nonzero polynomial $f\in V_k$ has at most $k-1$ zeros counted with multiplicity in $(0,\infty)$ and $[0,\infty]$, respectively; see \cref{sec_background_disconjugate}.
\begin{eg}\label{eg_intro}
Let $n := 3$, and let $V = (V_1, V_2)$ denote a generic element of $\Fl_3(\mathbb{R})$, represented by the matrix
$$
A := \begin{bmatrix}
1 & 0 & 0 \\
a & 1 & 0 \\
b & c & 1
\end{bmatrix} \quad (a,b,c\in\mathbb{R}).
$$
That is, $V_k$ (for $k = 1,2$) is the span of the first $k$ columns of $A$. Let us verify that \cref{flag_wronskian}\ref{flag_wronskian_positive} holds for $V$. The Pl\"{u}cker coordinates of $V$ are
$$
\Delta_1(V) = 1,\; \Delta_2(V) = a,\; \Delta_3(V) = b \quad \text{ and } \quad \Delta_{12}(V) = 1,\; \Delta_{13}(V) = c,\; \Delta_{23}(V) = ac - b.
$$
Therefore $V$ is totally positive if and only if
\begin{align}\label{eg_intro_inequalities}
a,\, b,\, c,\, ac-b > 0.
\end{align}

On the other hand, let $f_k\in\mathbb{R}[x]$ (for $k = 1,2$) be the polynomial corresponding to column $k$ of $A$, so that
$$
f_1(x) = 1 + ax + bx^2 \quad \text{ and } \quad f_2(x) = x + cx^2.
$$
Then
\begin{gather*}
\Wr(V_1) = \Wr(f_1) = 1 + ax + bx^2 =: h_1(x), \\
\Wr(V_2) = \Wr(f_1, f_2) = f_1f_2' - f_1'f_2 = 1 + 2cx + (ac-b)x^2 =: h_2(x).
\end{gather*}

We must show that \eqref{eg_intro_inequalities} holds if and only if $h_1$ and $h_2$ are nonzero on $[0,\infty]$. If \eqref{eg_intro_inequalities} holds, then $h_1$ and $h_2$ have positive coefficients, and hence they are positive on $[0,\infty]$. Conversely, suppose that $h_1$ and $h_2$ are nonzero on $[0,\infty]$. Since $h_1(0) = h_2(0) = 1$, they must be positive on $[0,\infty]$. Considering $x = \infty$, we get that $h_1$ and $h_2$ both have the expected degree $2$, and $b, ac - b > 0$. It remains to show that $a, c > 0$. Since $ac > b > 0$, we see that $a$ and $c$ have the same sign. Proceed by contradiction and suppose that $a,c \le 0$. Observe that $h_1$ is minimized at $x = -\frac{a}{2b}$, and $h_2$ is minimized at $x = -\frac{c}{ac-b}$. We obtain
$$
h_1\Big({-}\frac{a}{2b}\Big) = 1 - \frac{a^2}{4b} > 0 \quad \text{ and } \quad h_2\Big({-}\frac{c}{ac-b}\Big) = 1 - \frac{c^2}{ac-b} > 0.
$$
Therefore $a^2 < 4b$ and $c^2 < ac - b$. We now reach a contradiction:
\begin{gather*}
0 \le (a - 2c)^2 = a^2 - 4ac + 4c^2 < 4b - 4ac + 4(ac - b) = 0. \qedhere
\end{gather*}

\end{eg}

We outline the proof of part \ref{flag_wronskian_positive} of \cref{flag_wronskian}, whence we obtain part \ref{flag_wronskian_nonnegative} via a limiting argument. The forward direction is a consequence of the fact that the coefficients of $\Wr(V_k)$ are positively weighted sums of the Pl\"{u}cker coordinates of $V_k$. The reverse direction seems difficult to establish directly, as we have done in \cref{eg_intro} when $n=3$; rather, we argue as follows. First we note that $\Fl_n^{>0}$ is a connected component of the space of complete flags whose Wronskians are nonzero on $[0,\infty]$. Then we show that given $V\in\Fl_n(\mathbb{R})$ whose Wronskians are nonzero on $[0,\infty]$, there exists a path from $V$ to $\Fl_n^{>0}$; specifically, $V$ becomes totally positive upon replacing the variable $x$ with $x+t$ for $t > 0$ sufficiently large. This is based on the fact that the operator $x\mapsto x+t$ can be written as $\exp(t\frac{d}{dx})$, and is therefore compatible with total positivity.

We do not know how to generalize the statement of \cref{flag_wronskian} from $\Fl_n(\mathbb{R})$ to an arbitrary partial flag variety in $\mathbb{R}^n$; see \cref{eg_partial_flag}, and cf.\ \cref{positivity_conjecture}. The argument in the previous paragraph breaks down because the boundary of the totally nonnegative part of a partial flag variety (other than $\Fl_n(\mathbb{R})$ and $\mathbb{P}^{n-1}(\mathbb{R})$) does not have a simple description in terms of Wronskians.

A curious consequence of \cref{flag_wronskian} is that $\Fl_n^{\ge 0}$ and $\Fl_n^{>0}$ can be described as the semialgebraic subsets of $\Fl_n(\mathbb{R})$ where the coefficients of each Wronskian polynomial are nonnegative and positive, respectively (up to rescaling). Note that this description involves $O(n^3)$ inequalities, whereas the usual description using Pl\"{u}cker coordinates involves $2^n-2$ inequalities. Equivalently, we obtain a total nonnegativity test and a total positivity test for $\Fl_n(\mathbb{R})$ involving $O(n^3)$ functions. While total positivity tests for $\Fl_n(\mathbb{R})$ involving $O(n^2)$ functions are well-studied as part of the theory of cluster algebras (see \cite[Section 1.3]{fomin_williams_zelevinsky} and references therein), we do not know of any previous total nonnegativity tests for $\Fl_n(\mathbb{R})$ involving a subexponential number of fixed functions. In a separate paper, we will apply \cref{flag_wronskian} to give an efficient total nonnegativity test for $\GL_n(\mathbb{R})$.

We point out that Saldanha, Shapiro, and Shapiro \cite{saldanha_shapiro_shapiro21} have explored the connection between Wronskians of flags and total positivity. Also, Schechtman and Varchenko \cite[Theorem 4.4]{schechtman_varchenko20} have interpreted parametrizations of totally positive complete flags in terms of Wronskian polynomials. While these parametrizations do not play a role in our paper, it would be interesting to explore this connection further.

\subsection*{Conjectures in real Schubert calculus}
Given a system of polynomial equations over the real numbers, one often wishes to know how many of the solutions over the complex numbers are real. In general, we can usually say little more than that the nonreal solutions come in complex-conjugate pairs. In some situations, a lower bound on the number of real solutions can be given. In very special situations, all complex solutions are guaranteed to be real. Boris and Michael Shapiro discovered such a phenomenon in the Schubert calculus of Grassmannians.

Namely, for $0 \le k \le n$, let $\Gr_{k,n}(\mathbb{C})$ denote the {\itshape Grassmannian}, consisting of all $k$-dimensional subspaces $V$ of $\mathbb{C}^n$. Let $W_1, \dots, W_{k(n-k)}$ be sufficiently generic elements of $\Gr_{k,n}(\mathbb{C})$. Then the number of solutions $U\in\Gr_{n-k,n}(\mathbb{C})$ to the Schubert problem
\begin{align}\label{schubert_problem}
U\cap W_l\neq\{0\} \quad \text{ for } 1 \le l \le k(n-k)
\end{align}
is finite, and equals $d_{k,n} := \frac{1! 2! \cdots (k-1)!}{(n-k)! (n-k+1)! \cdots (n-1)!}(k(n-k))!$. Define the {\itshape rational normal curve}
\begin{align}\label{rational_normal_curve}
\rnc : \mathbb{C} \to \mathbb{C}^n, \quad x \mapsto \big(\textstyle\binom{n-1}{i-1}x^{n-i}\big)_{i=1}^n = \big(x^{n-1}, (n-1)x^{n-2}, \binom{n-1}{2}x^{n-3}, \dots, 1\big).
\end{align}
(Often one works instead with the curve $(1, x, x^2, \dots, x^{n-1})$, which is equivalent; the reason behind our alternative convention will be explained in \cref{sec_wronskian_formulation}.) Shapiro and Shapiro conjectured (cf.\ \cite{sottile00}) that if each $W_l$ is an osculating plane to $\rnc$ at a real point, then all solutions $U$ to the Schubert problem \eqref{schubert_problem} are real (i.e.\ have a basis of real vectors). Their conjecture was verified when $k=2$ by Eremenko and Gabrielov \cite{eremenko_gabrielov02}, and in general by Mukhin, Tarasov, and Varchenko \cite{mukhin_tarasov_varchenko09a,mukhin_tarasov_varchenko09b}, who also showed that the upper bound $d_{k,n}$ on the number of solutions is always obtained. A different proof was recently given by Levinson and Purbhoo \cite[Corollary 1.5]{levinson_purbhoo21}.

Let $\mathbb{C}[x]_{\le n-1}$ denote the subspace of $\mathbb{C}[x]$ of all polynomials of degree at most $n-1$, which we identify with $\mathbb{C}^n$ via \eqref{polynomial_identification} (with $\mathbb{R}$ replaced by $\mathbb{C}$). While not immediately apparent, a standard transformation (see \cref{sec_wronskian_formulation}) allows one to reformulate the Schubert problem above dually in terms of Wronskians. In particular, we have:
\begin{thm}[{Mukhin, Tarasov, and Varchenko \cite[Corollary 6.3]{mukhin_tarasov_varchenko09b}}]\label{reality}
Let $0 \le k \le n$, and let $X\subseteq\mathbb{R}$ consist of $k(n-k)$ distinct points. Then there exist precisely $d_{k,n}$ elements $V\in\Gr_{k,n}(\mathbb{C})$ such that the zero set of $\Wr(V)$ is $X$. Moreover, each such $V$ is real, i.e., it has a basis of real polynomials. In particular, if $V\in\Gr_{k,n}(\mathbb{C})$ such that all complex zeros of $\Wr(V)$ are real, then $V$ is real.
\end{thm}

Around 2003, Frank Sottile formulated a conjecture generalizing \cref{reality}, known as the {\itshape secant conjecture}. It appeared in \cite{ruffo_sivan_soprunova_sottile06}, and was extensively tested numerically as part of a thorough experiment led by Sottile \cite{garcia-puente_hein_hillar_martin_del_campo_ruffo_sottile_teitler12}, as described in \cite{hillar_garcia-puente_martin_del_campo_ruffo_teitler_johnson_sottile10}. It states in particular that the solutions to the Schubert problem \eqref{schubert_problem} all remain real if each osculating plane $W_l$ to $\rnc$ is replaced with a plane spanned by $\rnc(x_1), \dots, \rnc(x_k)$ for any real points $x_1, \dots, x_k$, such that the points chosen for each $W_l$ lie in $k(n-k)$ disjoint intervals. (This statement is the special case of the secant conjecture for divisor Schubert conditions, i.e., Schubert conditions of codimension one; the general conjecture involves Schubert conditions of arbitrary codimension, which we do not consider in this paper.) Eremenko \cite{eremenko15} showed that this case of the secant conjecture is implied by \cref{reality} and the following conjecture. The case $k=2$ of both conjectures is a consequence of work of Eremenko, Gabrielov, Shapiro, and Vainshtein \cite[Section 3]{eremenko_gabrielov_shapiro_vainshtein06} (cf.\ \cite[p.\ 341]{eremenko15}).
\begin{conj}[Eremenko \cite{eremenko15, eremenko19}]\label{disconjugacy_conjecture}
Let $V\in\Gr_{k,n}(\mathbb{R})$. Suppose that all complex zeros of $\Wr(V)$ are real, and let $I\subseteq\mathbb{R}$ be any interval on which $\Wr(V)$ is nonzero. Then every nonzero polynomial $f\in V$ has at most $k-1$ zeros in $I$.
\end{conj}

We now conjecture that certain of the Schubert problems considered above have all their solutions not only real, but totally nonnegative or totally positive. Namely, we say that $V\in\Gr_{k,n}(\mathbb{R})$ is {\itshape totally nonnegative} (respectively, {\itshape totally positive}) if all its Pl\"{u}cker coordinates are nonnegative (respectively, positive), up to rescaling. We also extend the definition of $\rnc$ to $\mathbb{P}^1(\mathbb{C}) = \mathbb{C}\cup\{\infty\}$; see \cref{defn_plane_span}. We make the following totally positive analogue of the secant conjecture:
\begin{conj}\label{positive_secant_conjecture}
Let $0 \le k \le n$, and let $I_1, \dots, I_{k(n-k)}$ be pairwise disjoint intervals of $\mathbb{P}^1(\mathbb{R})$. For $1 \le l \le k(n-k)$, let $X_l$ be a multiset of $k$ points contained in $I_l$, and let $W_l\in\Gr_{k,n}(\mathbb{R})$ be spanned by $\rnc(x), \rnc'(x), \dots, \rnc^{(p-1)}(x)$ for all $x\in X_l$, where $p$ denotes the multiplicity of $x$ in $X_l$.
\begin{enumerate}[label=(\roman*), leftmargin=*, itemsep=2pt]
\item\label{positive_secant_conjecture_nonnegative} If $I_l\subseteq [0,\infty]$ for $1 \le l \le k(n-k)$, then the Schubert problem \eqref{schubert_problem} has $d_{k,n}$ distinct solutions $U\in\Gr_{n-k,n}(\mathbb{C})$, which are all real and totally nonnegative.
\item\label{positive_secant_conjecture_positive} If $I_l\subseteq (0,\infty)$ for $1 \le l \le k(n-k)$, then the Schubert problem \eqref{schubert_problem} has $d_{k,n}$ distinct solutions $U\in\Gr_{n-k,n}(\mathbb{C})$, which are all real and totally positive.
\end{enumerate}
\end{conj}

A limiting argument implies that parts \ref{positive_secant_conjecture_nonnegative} and \ref{positive_secant_conjecture_positive} of \cref{positive_secant_conjecture} for $\Gr_{n-k,n}(\mathbb{C})$ are equivalent. We also note that \cref{positive_secant_conjecture} implies the case of the secant conjecture stated above. This is due to an action of $\SL_2(\mathbb{R})$ (see \cref{sec_SL_action}) which allows us to assume, without loss of generality, that all the intervals appearing in the secant conjecture are contained in $(0,\infty)$.

A special case of \cref{positive_secant_conjecture} is when each multiset $X_l$ consists of a single point $x_l$ of multiplicity $k$, so that $W_l$ is the osculating plane to $\rnc$ at $x_l$. Then \cref{positive_secant_conjecture} has the following dual formulation, which is a totally positive analogue of \cref{reality}:
\begin{conj}\hspace*{-4pt}\footnote{\cref{positivity_conjecture}\ref{positivity_conjecture_nonnegative} was independently posed by Evgeny Mukhin and Vitaly Tarasov in 2017. I thank Chris Fraser for informing me of this.}\label{positivity_conjecture}
Let $V\in\Gr_{k,n}(\mathbb{R})$.
\begin{enumerate}[label=(\roman*), leftmargin=*, itemsep=2pt]
\item\label{positivity_conjecture_nonnegative} If all complex zeros of $\Wr(V)$ lie in the interval $[-\infty, 0]$, then $V$ is totally nonnegative.
\item\label{positivity_conjecture_positive} If all complex zeros of $\Wr(V)$ lie in the interval $(-\infty, 0)$, then $V$ is totally positive.
\end{enumerate}
Above, $\Wr(V)$ is considered to have a zero at $-\infty$ when its degree is less than $k(n-k)$.
\end{conj}

Note that it is equivalent to replace $\Gr_{k,n}(\mathbb{R})$ by $\Gr_{k,n}(\mathbb{C})$ in \cref{positivity_conjecture}, by \cref{reality}. \cref{positivity_conjecture} holds when $k=1$; it states that if all complex zeros of a polynomial $f\in\mathbb{R}[x]$ are nonpositive (respectively, negative), then up to rescaling, all coefficients of $f$ are nonnegative (respectively, positive). For an example, see \cref{eg_positivity_conjecture}, where we verify \cref{positivity_conjecture}\ref{positivity_conjecture_positive} for $\Gr_{2,4}(\mathbb{R})$. The case $n=5$ of \cref{positivity_conjecture} was proved by Fraser \cite{fraserb}, and we have checked \cref{positivity_conjecture} by computer for several instances with $n=6$.

A limiting argument implies that parts \ref{positivity_conjecture_nonnegative} and \ref{positivity_conjecture_positive} of \cref{positivity_conjecture} for $\Gr_{k,n}(\mathbb{R})$ are equivalent. Also, \cref{positivity_conjecture} holds for $\Gr_{k,n}(\mathbb{R})$ if and only if it holds for $\Gr_{n-k,n}(\mathbb{R})$. This is due to the existence of a certain bilinear pairing \eqref{pairing} on $\mathbb{R}^n$, such that the map $V\mapsto V^\perp$ preserves both the Wronskian and the collection of signs of Pl\"{u}cker coordinates. In particular, \cref{positivity_conjecture} holds for $\Gr_{n-1,n}(\mathbb{R})$, since it holds for $\Gr_{1,n}(\mathbb{R})$.

Our second main result is that the three conjectures stated above are all equivalent to each other:
\begin{thm}\label{conjectures_equivalent}
Let $0 \le k \le n$.
\begin{enumerate}[label=(\roman*), leftmargin=*, itemsep=2pt]
\item\label{conjectures_equivalent_disconjugacy} \cref{positivity_conjecture} holds for $\Gr_{k,n}(\mathbb{R})$ if and only if \cref{disconjugacy_conjecture} holds for both $\Gr_{k,n}(\mathbb{R})$ and $\Gr_{n-k,n}(\mathbb{R})$.
\item\label{conjectures_equivalent_secant} \cref{positivity_conjecture} holds for $\Gr_{k,n}(\mathbb{R})$ if and only if \cref{positive_secant_conjecture} holds for $\Gr_{n-k,n}(\mathbb{C})$.
\end{enumerate}

\end{thm}

In particular, \cref{positivity_conjecture} implies the secant conjecture for divisor Schubert conditions. Another consequence of \cref{conjectures_equivalent}\ref{conjectures_equivalent_secant} is that \cref{positive_secant_conjecture} holds for $\Gr_{n-k,n}(\mathbb{C})$ if and only if it holds for $\Gr_{k,n}(\mathbb{C})$. This is in contrast to the secant conjecture, where the statements for $\Gr_{n-k,n}(\mathbb{C})$ and $\Gr_{k,n}(\mathbb{C})$ are not known to imply each other (see \cite[Section 2.3]{garcia-puente_hein_hillar_martin_del_campo_ruffo_sottile_teitler12}).

The proof of \cref{conjectures_equivalent}\ref{conjectures_equivalent_disconjugacy} uses \cref{flag_wronskian}, the $\SL_2(\mathbb{R})$-action, the bilinear pairing \eqref{pairing}, and classical results on Chebyshev and disconjugate systems of functions. \cref{conjectures_equivalent}\ref{conjectures_equivalent_secant} then follows from the same argument used by Eremenko \cite{eremenko15} to show that \cref{disconjugacy_conjecture} implies the secant conjecture for divisor Schubert conditions. We point out that in reducing \cref{positivity_conjecture}\ref{positivity_conjecture_positive} to \cref{disconjugacy_conjecture} for a given subspace $V$, we call on \cref{disconjugacy_conjecture} applied to both $V$ and $V^\perp$. In particular, although \cref{disconjugacy_conjecture} holds for $\Gr_{2,n}(\mathbb{R})$, it is open for $\Gr_{n-2,n}(\mathbb{R})$, and so we are not able to conclude that \cref{positivity_conjecture} holds for $\Gr_{2,n}(\mathbb{R})$. Indeed, \cref{positivity_conjecture} is open for $\Gr_{2,n}(\mathbb{R})$.

\subsection*{Outline}
In \cref{sec_background_wronskians}, we recall some background on total positivity and Wronskians, and develop some preliminary tools. In \cref{sec_wronskians}, we prove \cref{flag_wronskian}. In \cref{sec_disconjugate}, we recall some background on Chebyshev and disconjugate systems, and then prove \cref{conjectures_equivalent} and the preceding claims.

\subsection*{Acknowledgments}
I thank Alexandre Eremenko, Sergey Fomin, Chris Fraser, Kevin Purbhoo, and Frank Sottile for valuable discussions, and the reviewers for helpful feedback.

\section{Flag varieties and Wronskians}\label{sec_background_wronskians}

\noindent In this section we recall some background on flag varieties, total positivity, and Wronskians of polynomials. Much of the material is known; for further details on Wronskians and real Schubert calculus, we refer to \cite{sottile11,purbhoo10}. For the sake of exposition, we will work over the complex numbers.

\subsection{Notation}\label{sec_notation}
Let $\mathbb{N} := \{0, 1, 2, \dots\}$. For $n\in\mathbb{N}$, we let $[n] := \{1, \dots, n\}$, and for $0 \le k \le n$, we let $\binom{[n]}{k} := \{I\subseteq [n] : |I| = k\}$. We define a partial order on $\binom{[n]}{k}$, called the {\itshape Gale order}, by
$$
I \le J \hspace*{2pt}\Leftrightarrow\hspace*{2pt} i_1 \le j_1, \dots, i_k \le j_k \quad \text{ for all } I = \{i_1 < \cdots < i_k\}, J = \{j_1 < \cdots < j_k\} \in \textstyle\binom{[n]}{k}.
$$
For a finite subset $I\subseteq\mathbb{N}$, we let $\sumof{I}$ denote the sum of the elements of $I$.

Given an $m\times n$ matrix $A$, for $I\subseteq [m]$ and $J\subseteq [n]$, we let $A_{I,J}$ denote the submatrix of $A$ using rows $I$ and columns $J$. A {\itshape minor} of $A$ is the determinant of a square submatrix. We have the {\itshape Cauchy--Binet identity} for the minors of a product of two matrices (see e.g.\ \cite[I.(14)]{gantmacher59}): if $A$ is an $m\times n$ matrix, $B$ is an $n\times p$ matrix, and $1 \le k \le m,p$, then
\begin{align}\label{cauchy-binet}
\det((AB)_{I,J}) = \sum_{K\in\binom{[n]}{k}}\det(A_{I,K})\det(B_{K,J}) \quad \text{ for all } I\in\textstyle\binom{[m]}{k} \text{ and } J\in\binom{[p]}{k}.
\end{align}

We say that $A$ is a {\itshape totally positive matrix} if all its minors are positive. We say that an $n\times n$ matrix $A$ is a {\itshape totally positive upper-triangular matrix} if $A$ is upper-triangular and $\det(A_{I,J}) > 0$ for all $1 \le k \le n$ and $I\le J$ in $\binom{[n]}{k}$; that is, all the minors of $A$ are positive except those which are zero due to upper-triangularity. We similarly define the notion of a {\itshape totally positive lower-triangular matrix}.

Recall that $\mathbb{C}[x]_{\le n-1}$ denotes the $n$-dimensional subspace of $\mathbb{C}[x]$ consisting of all polynomials with degree at most $n-1$, which we identify with $\mathbb{C}^n$ via \eqref{polynomial_identification} (with $\mathbb{R}$ replaced by $\mathbb{C}$). Given a nonzero polynomial $f\in\mathbb{C}[x]_{\le n-1}$, we view the zeros of $f$ as being a multiset of $\mathbb{C}\cup\{\infty\}$ of size $n-1$, where the multiplicity of the zero at $\infty$ is $n - 1 - \deg(f)$. Note that the multiset of zeros only depends on $f$ modulo rescaling by $\mathbb{C}^{\times}$. We will sometimes denote $\infty$ by $-\infty$. We also regard $\mathbb{C}\cup\{\infty\}$ as $\mathbb{P}^1(\mathbb{C})$, where the point $x\in\mathbb{C}$ corresponds to $(x : 1)\in\mathbb{P}^1(\mathbb{C})$, and $\infty$ corresponds to $(1 : 0)\in\mathbb{P}^1(\mathbb{C})$. We say that {\itshape all zeros of $f$ are real} if its $n-1$ zeros lie in $\mathbb{R}\cup\{\infty\}$, or equivalently, in $\mathbb{P}^1(\mathbb{R})$.

\subsection{Flag varieties}\label{sec_grassmannian}
We now introduce Grassmannians and complete flag varieties, and define their totally nonnegative and totally positive parts.

\begin{defn}\label{defn_flag}
Let $n\in\mathbb{N}$. For $0 \le k \le n$, we define the {\itshape Grassmannian $\Gr_{k,n}(\mathbb{R})$} as the set of all $k$-dimensional subspaces of $\mathbb{R}^n$. Given $V\in\Gr_{k,n}(\mathbb{R})$, we say that an $n\times k$ matrix $A$ {\itshape represents $V$} if the columns of $A$ form a basis of $V$. (The matrix $A$ is unique modulo invertible column operations.) For $I\in\binom{[n]}{k}$, we let $\Delta_I(V)$ denote $\det(A_{I,[k]})$, called a {\itshape Pl\"{u}cker coordinate}. The minors $(\Delta_I(V))_{I\in\binom{[n]}{k}}$ are well-defined up to simultaneous rescaling, and provide projective coordinates on $\Gr_{k,n}(\mathbb{R})$.

We define the {\itshape complete flag variety} $\Fl_n(\mathbb{R})$ as the set of tuples $W = (W_1, \dots, W_{n-1})$, such that $W_k\in\Gr_{k,n}(\mathbb{R})$ for $1 \le k \le n-1$, and $W_1 \subset \cdots \subset W_{n-1}$. We say that an $n\times n$ matrix $B$ {\itshape represents} $W$ if for $1 \le k \le n-1$, the first $k$ columns of $B$ form a basis of $W_k$. 

We also make the corresponding definitions with $\mathbb{R}$ replaced by $\mathbb{C}$. We regard $\Gr_{k,n}(\mathbb{R})$ as the subset of $\Gr_{k,n}(\mathbb{C})$ of all subspaces which have a basis of real vectors (equivalently, which have real Pl\"{u}cker coordinates). We similarly regard $\Fl_n(\mathbb{R})$ as a subset of $\Fl_n(\mathbb{C})$.
\end{defn}

\begin{eg}\label{eg_flag}
Let $V\in\Gr_{2,4}(\mathbb{C})$ be represented by the matrix
$$
\begin{bmatrix}1 & 0 \\ 0 & 1 \\ a & b \\ c & d\end{bmatrix}.
$$
Then the Pl\"{u}cker coordinates of $V$ are
\begin{align*}
\hspace*{36pt}&\Delta_{12}(V) = 1, &&\Delta_{13}(V) = b, &&\Delta_{14}(V) = d,\hspace*{36pt} \\
\hspace*{36pt}&\Delta_{23}(V) = -a, &&\Delta_{24}(V) = -c, &&\Delta_{34}(V) = ad-bc.\hspace*{36pt}\qedhere
\end{align*}

\end{eg}

For an example in the case of $\Fl_3(\mathbb{R})$, see \cref{eg_intro}.

\begin{defn}\label{defn_nonnegative}
Let $n\in\mathbb{N}$. For $0 \le k \le n$, we say that $V\in\Gr_{k,n}(\mathbb{R})$ is {\itshape totally nonnegative} if all Pl\"{u}cker coordinates of $V$ are nonnegative (up to rescaling), i.e., $V$ can be represented by an $n\times k$ matrix whose maximal minors are nonnegative. We let $\Gr_{k,n}^{\ge 0}$ denote the set of all totally nonnegative $V$.

Let $W = (W_1, \dots, W_{n-1})\in\Fl_n(\mathbb{R})$. We say that $W$ is {\itshape totally nonnegative} if $W_k$ is totally nonnegative in $\Gr_{k,n}(\mathbb{R})$ for $1 \le k \le n-1$, i.e., $W$ can be represented by an $n\times n$ matrix whose left-justified minors (i.e.\ those which use an initial set of columns) are nonnegative. We denote the totally nonnegative subset of $\Fl_n(\mathbb{R})$ by $\Fl_n^{\ge 0}$.

We make the same definitions with ``nonnegative'' replaced by ``positive'', and obtain the totally positive subsets $\Gr_{k,n}^{>0}$ and $\Fl_n^{>0}$.
\end{defn}

\begin{eg}\label{eg_nonnegative}
Let $V\in\Gr_{2,4}(\mathbb{C})$ be as in \cref{eg_flag}. Then $V$ is totally nonnegative if and only if
$$
b, d, -a, -c, ad-bc \ge 0,
$$
and $V$ is totally positive if and only if all the inequalities above hold strictly.
\end{eg}

\begin{rmk}\label{two_notions}
The definitions of the totally nonnegative and totally positive parts of $\Fl_n(\mathbb{R})$ and $\Gr_{k,n}(\mathbb{R})$ in \cref{defn_flag} are different from, but equivalent to, the original definitions of Lusztig \cite{lusztig94,lusztig98}. Our definitions for $\Gr_{k,n}(\mathbb{R})$ agree with those of Postnikov \cite{postnikov07}. In the case of $\Gr_{k,n}(\mathbb{R})$, the equivalence is essentially \cref{flag_extension}. We refer to \cite[Section 1]{bloch_karp2} for further discussion and references.
\end{rmk}

We have the following characterization of $\Fl_n^{>0}$ due to Lusztig \cite{lusztig94}. For a closely related result, see \cite{shapiro_shapiro95}.
\begin{prop}[{Lusztig \cite[Proposition 8.14]{lusztig94}}]\label{connected_component_plucker}
The totally positive flag variety $\Fl_n^{>0}$ is a connected component inside $\Fl_n(\mathbb{R})$ of the subset where
\begin{gather*}
\Delta_{[k]}, \Delta_{[n]\setminus [n-k]} \neq 0 \quad \text{ for } 1 \le k \le n-1.
\end{gather*}

\end{prop}

We will need at least one of the following two results for the proof of \cref{conjectures_equivalent}.
\begin{thm}[Rietsch \cite{rietsch09}]\label{flag_extension}
Let $V\in\Gr_{k,n}^{\ge 0}$. Then there exists a totally nonnegative complete flag $(W_1, \dots, W_{n-1})\in\Fl_n^{\ge 0}$ such that $W_k = V$.
\end{thm}

\begin{thm}[{Gantmakher and Krein \cite[Theorem V.3]{gantmaher_krein50}}]\label{gantmakher_krein}
Let $V\in\Gr_{k,n}(\mathbb{R})$. Then $V$ is totally nonnegative if and only if each vector in $V$ changes sign at most $k-1$ times, when viewed as a sequence of $n$ numbers and ignoring any zeros.
\end{thm}

\subsection{Wronski map on the Grassmannian}\label{sec_wronski_map}
We now define the Wronskian of an element $V\in\Gr_{k,n}(\mathbb{C})$, by viewing $V$ as a subspace of $\mathbb{C}[x]_{\le n-1}$.
\begin{defn}\label{defn_wronski_map}
Given $V\in\Gr_{k,n}(\mathbb{C})$, we define the {\itshape Wronskian} $\Wr(V)$ of $V$ by
$$
\Wr(V) := \Wr(f_1, \dots, f_k) \in\mathbb{P}(\mathbb{C}[x]_{\le k(n-k)}),
$$
where $f_1, \dots, f_k\in\mathbb{C}[x]_{\le n-1}$ is a $\mathbb{C}$-linear basis of $V$, and the Wronskian of $k$ functions is defined in \eqref{wronskian_formula}. We regard $\Wr(V)$ as an element of $\mathbb{C}[x]_{\le k(n-k)}$ modulo rescaling by $\mathbb{C}^\times$, i.e., as an element of $\mathbb{P}(\mathbb{C}[x]_{\le k(n-k)})$, so that it does not depend on the choice of basis $f_1, \dots, f_k$. To see that $\Wr(V)$ has degree at most $k(n-k)$, note that we can choose a basis of $V$ such that $f_i$ has degree at most $n-i$, for $1 \le i \le k$. If $\deg(\Wr(V)) < k(n-k)$, then we regard $\Wr(V)$ as having a zero at $\pm\infty$ of order $k(n-k) - \deg(\Wr(V))$.
\end{defn}

We have an explicit formula for $\Wr(V)$ in terms of the Pl\"{u}cker coordinates of $V$.

\begin{defn}\label{defn_vandermonde_constant}
Given $I\subseteq [n]$, let $\vand{I} := \displaystyle\frac{1}{1!\hspace*{1pt}2!\hspace*{1pt}\cdots (k-1)!}\prod_{i,j\in I, \, i < j}(j-i) \in \mathbb{Z}_{>0}$.
\end{defn}

For example, we have $\vand{124} = \frac{1}{1!\hspace*{1pt}2!}(2-1)(4-1)(4-2) = 3$. If $|I| = 1$, then $\vand{I} = 1$.
\begin{prop}[{Purbhoo \cite[Proposition 2.3]{purbhoo10}}]\label{wronskian_pluckers}
Let $V\in\Gr_{k,n}(\mathbb{C})$. Then
\begin{gather*}
\Wr(V) = \sum_{I\in\binom{[n]}{k}}\vand{I}\Delta_I(V)x^{\sumof{I} - \binom{k+1}{2}}.
\end{gather*}

\end{prop}

\begin{eg}\label{eg_wronskian_pluckers}
For $V\in\Gr_{2,4}(\mathbb{C})$, we have
\begin{gather*}
\Wr(V) = \Delta_{12}(V) + 2\Delta_{13}(V)x + (3\Delta_{14}(V) + \Delta_{23}(V))x^2 + 2\Delta_{24}(V)x^3 + \Delta_{34}(V)x^4.\qedhere
\end{gather*}

\end{eg}

\subsection{Bilinear pairing}\label{sec_pairing}
We now define a bilinear pairing on $\mathbb{C}^n$, which induces a duality between $\Gr_{k,n}(\mathbb{C})$ and $\Gr_{n-k,n}(\mathbb{C})$ compatible with both total positivity and the Wronski map. We refer to \cite[Section 2.3]{garcia-puente_hein_hillar_martin_del_campo_ruffo_sottile_teitler12} for a related discussion of Grassmannian duality.
\begin{defn}\label{defn_pairing}
We define the bilinear pairing $\langle\cdot,\cdot\rangle$ on $\mathbb{C}^n$ by
\begin{align}\label{pairing}
\langle a,b\rangle := \sum_{i=1}^n(-1)^{i-1}\frac{a_ib_{n+1-i}}{\binom{n-1}{i-1}} \quad \text{ for } a,b\in\mathbb{C}^n.
\end{align}
Given $V\in\Gr_{k,n}(\mathbb{C})$, we let $V^\perp\in\Gr_{n-k,n}(\mathbb{C})$ denote the space perpendicular to $V$ under $\langle\cdot,\cdot\rangle$:
$$
V^\perp := \{b\in\mathbb{R}^n : \langle a,b\rangle = 0 \text{ for all } a\in V\}.
$$
Also, for $I \in\binom{[n]}{k}$, we define $I^\perp\in\binom{[n]}{n-k}$ by
\begin{gather*}
I^\perp := \{n+1 - i : i\in [n]\setminus I\}.
\end{gather*}

\end{defn}

\begin{eg}\label{eg_pairing}
Let $V\in\Gr_{2,4}(\mathbb{C})$ be represented by the matrix
$$
\begin{bmatrix}1 & 0 \\ 0 & 1 \\ a & b \\ c & d\end{bmatrix}.
$$
Then $V^\perp\in\Gr_{2,4}(\mathbb{C})$ is represented by the matrix
\begin{gather*}
\begin{bmatrix}1 & 0 \\ 0 & 3 \\ -3d & 3b \\ c & -a\end{bmatrix}.\qedhere
\end{gather*}

\end{eg}

We recall the constants $\vand{I}$ introduced in \cref{defn_vandermonde_constant}.
\begin{lem}\label{perpendicular_sum}
Let $I\in\binom{[n]}{k}$.
\begin{enumerate}[label=(\roman*), leftmargin=*, itemsep=6pt]
\item\label{perpendicular_sum_sum} We have $\sumof{I} - \binom{k+1}{2} = \sumof{I^\perp} - \binom{n-k+1}{2}$.
\item\label{perpendicular_sum_product} We have $\displaystyle\frac{\prod_{i\in I}(i-1)!}{1!\hspace{1pt}2!\hspace*{1pt}\cdots (k-1)!\hspace*{1pt}\vand{I}} = \frac{\prod_{i\in I^\perp}(i-1)!}{1!\hspace*{1pt}2!\hspace*{1pt}\cdots (n-k-1)!\hspace*{1pt}\vand{I^\perp}}$.
\end{enumerate}

\end{lem}

\begin{proof}
\ref{perpendicular_sum_sum} We have $\textstyle\sumof{I^\perp} = (n-k)(n+1) - (\sumof{[n]} - \sumof{I}) = \textstyle\binom{n-k+1}{2} - \binom{k+1}{2} + \sumof{I}$.\vspace*{2pt}

\ref{perpendicular_sum_product} We can show that both sides equal $\displaystyle\prod_{i\in [n]\setminus I,\, j\in I, \, i < j}(j-i)$.
\end{proof}

The following result follows from \cite[Lemma 1.11(ii)]{karp17}, to which we refer for further discussion and references.
\begin{lem}\label{perpendicular_pluckers}
Let $V\in\Gr_{k,n}(\mathbb{C})$. Then
$$
\vand{I}\Delta_I(V) = \vand{I^\perp}\Delta_{I^\perp}(V^\perp) \quad \text{ for all } I\in\textstyle\binom{[n]}{k}.
$$
In particular, $V$ is totally nonnegative if and only if $V^\perp$ is totally nonnegative, and $V$ is totally positive if and only if $V^\perp$ is totally positive.
\end{lem}

\begin{proof}
Let $W\in\Gr_{n-k,n}(\mathbb{C})$ be the space perpendicular to $V$ under the standard bilinear pairing $(a,b) := \sum_{i=1}^na_ib_i$. By \cite[Lemma 1.11(ii)]{karp17}, we have
$$
\Delta_I(V) = (-1)^{\sumof{I}}\Delta_{[n]\setminus I}(W) \quad \text{ for all } I\in\textstyle\binom{[n]}{k}.
$$
Now observe that given an $n\times (n-k)$ matrix representing $W$, we obtain an $n\times (n-k)$ matrix representing $V^\perp$ by reversing the order of the rows, and then multiplying row $i$ by $(-1)^{i-1}\binom{n-1}{i-1}$ for $1 \le i \le n$. Therefore
$$
\Delta_I(V) = \frac{(-1)^{\sumof{I} + \sumof{I^\perp}}}{\prod_{i\in I^\perp}\textstyle\binom{n-1}{i-1}}\Delta_{I^\perp}(V^\perp) \quad \text{ for all } I\in\textstyle\binom{[n]}{k}.
$$
It now suffices to show that $\displaystyle\frac{(-1)^{\sumof{I} + \sumof{I^\perp}}}{\prod_{i\in I^\perp}\textstyle\binom{n-1}{i-1}}$ equals $\displaystyle\frac{\vand{I^\perp}}{\vand{I}}$, up to multiplication by a constant which depends only on $k$ and $n$ (but not $I$). This follows from \cref{perpendicular_sum}.
\end{proof}

\begin{cor}\label{perpendicular_wronskian}
We have $\Wr(V) = \Wr(V^\perp)$ for all $V\in\Gr_{k,n}(\mathbb{C})$.
\end{cor}

\begin{proof}
This follows from \cref{wronskian_pluckers}, \cref{perpendicular_sum}\ref{perpendicular_sum_sum}, and \cref{perpendicular_pluckers}.
\end{proof}

For example, we can verify that for the element $V\in\Gr_{2,4}(\mathbb{C})$ from \cref{eg_pairing}, we have $\Wr(V) = \Wr(V^\perp)$ (and observe that $V\neq V^\perp$ unless $a = -3d$).

\subsection{Wronskian formulation}\label{sec_wronskian_formulation}
We now explain how to translate the Schubert problem \eqref{schubert_problem} considered by Shapiro and Shapiro into a statement about Wronskian polynomials. While this reformulation is well-known, we caution that some of the details of our exposition are non-standard, since we use the bilinear pairing \eqref{pairing}. For an alternative exposition based on the standard bilinear pairing $(a,b) := \sum_{i=1}^na_ib_i$, see \cite[Section 10.1]{sottile11} or \cite[Section 2.3]{garcia-puente_hein_hillar_martin_del_campo_ruffo_sottile_teitler12}.
\begin{defn}\label{defn_plane_span}
Let $0 \le k \le n$, and recall the rational normal curve $\rnc$ from \eqref{rational_normal_curve}. Let $X\subseteq\mathbb{P}^1(\mathbb{C})$ be a multiset of size $k$, consisting of the elements $x_1, \dots, x_m\in\mathbb{P}^1(\mathbb{C})$ with multiplicities $p_1, \dots, p_m$, where $p_1 + \cdots + p_m = k$. Define $\Rnc{X}\in\Gr_{k,n}(\mathbb{C})$ to be the span of $\rnc(x_i), \rnc'(x_i), \dots, \rnc^{(p_i-1)}(x_i)$ for all $1 \le i \le m$, where if $x_i = \infty$, we define $\rnc^{(j)}(\infty)$ to be the unit vector with a $1$ in entry $j+1$. (It will follow from \cref{plane_span_zeros} that $\Rnc{X}$ has dimension $k$.) We also define $\zeros{X}\in\Gr_{n-k,n}(\mathbb{C})$ to be the subspace of $\mathbb{C}[x]_{\le n-1}$ of polynomials with a zero at $x_i$ of multiplicity at least $p_i$, for all $1 \le i \le m$. In particular, if $X$ consists of a single element $x$ of multiplicity $k$, then $\Rnc{X}$ is the osculating plane to $\rnc$ at $x$, and $\zeros{X}$ is the subspace of $\mathbb{C}[x]_{\le n-1}$ of polynomials with a zero of multiplicity at least $k$ at $x$.
\end{defn}

Our definition of the rational normal curve in \eqref{rational_normal_curve} was chosen to be compatible with the bilinear pairing \eqref{pairing}: 
\begin{lem}\label{plane_span_zeros}
Let $0 \le k \le n$, and let $X\subseteq\mathbb{P}^1(\mathbb{C})$ be a multiset of size $k$. Then
\begin{gather*}
(\Rnc{X})^\perp = \zeros{-X}.
\end{gather*}

\end{lem}

\begin{proof}
This follows from the fact that for all $f = \sum_{i=1}^na_ix^{i-1}\in\mathbb{C}[x]_{\le n-1}$ and $0 \le j \le n-1$, we have
$$
\langle\gamma^{(j)}(x), f\rangle = (-1)^{n-j-1}f^{(j)}(-x) \quad \text{ for all } x\in\mathbb{C},
$$
and $\langle\gamma^{(j)}(\infty), f\rangle$ equals $a_{n-j}$ modulo a nonzero constant.
\end{proof}

\begin{eg}\label{eg_plane_span}
Let $k := 3$ and $n := 5$, and let $X := \{0, 0, 1\}$. Then $\Rnc{X}\in\Gr_{3,5}(\mathbb{C})$ is spanned by
$$
(0,0,0,0,1), \quad (0,0,0,4,0), \quad \text{ and } \quad (1,4,6,4,1).
$$
We can check that $(\Rnc{X})^\perp$ consists precisely of those elements of $\mathbb{C}[x]_{\le 4}$ with zeros at $0$ and $-1$ of multiplicities at least $2$ and $1$, respectively, in agreement with \cref{plane_span_zeros}.
\end{eg}

\begin{cor}\label{dual_problem}
Let $0 \le k \le n$, let $X\subseteq\mathbb{P}^1(\mathbb{C})$ be a multiset of size $k$, and let $V\in\Gr_{k,n}(\mathbb{C})$. Then
$$
V^\perp\cap\Rnc{X} \neq \{0\} \quad \Leftrightarrow \quad V\cap\zeros{-X}\neq\{0\}.
$$
If $X$ consists of a single element $x$ of multiplicity $k$, then both conditions above are equivalent to $\Wr(V)$ being zero at $-x$.
\end{cor}

\begin{proof}
This follows from \cref{plane_span_zeros}.
\end{proof}

Using \cref{dual_problem}, we can translate the Schubert problem \eqref{schubert_problem} into a dual formulation by writing $V = U^\perp$, as we did implicitly in \cref{sec_introduction}. Recall from \cref{perpendicular_pluckers} that the map $\cdot^\perp$ preserves reality, total nonnegativity, and total positivity. We emphasize that in this translation, the Grassmannian $\Gr_{n-k,n}(\mathbb{C})$ (containing $U$) appearing in the Schubert problem \eqref{schubert_problem} becomes identified with the Grassmannian $\Gr_{k,n}(\mathbb{C})$ (containing $V$) of subspaces of $\mathbb{C}[x]_{\le n-1}$ of dimension $k$. In this sense, the parameter $k$ has a consistent usage throughout the statements of \cref{reality}, \cref{disconjugacy_conjecture}, \cref{positive_secant_conjecture}, and \cref{positivity_conjecture}.

\subsection{Action of \texorpdfstring{$\SL_2(\mathbb{C})$}{SL(2)}}\label{sec_SL_action}
We recall an action of $\SL_2(\mathbb{C})$ on some of the objects defined above. We follow \cite[Section 2.1]{purbhoo10}, but employ a different sign convention.
\begin{defn}\label{defn_SL_action}
Let $\SL_2(\mathbb{C})$ act on $\mathbb{P}^1(\mathbb{C})$ in the usual way:
$$
\begin{bmatrix}
a & b \\
c & d
\end{bmatrix}\cdot x := \frac{ax + b}{cx + d} \quad \text{ for all } x\in\mathbb{P}^1(\mathbb{C}).
$$
In particular, $\SL_2(\mathbb{C})$ acts on the multiset of zeros of any polynomial. We define an action of $\SL_2(\mathbb{C})$ on $\mathbb{C}[x]_{\le n-1}$ compatible with the action on zeros:
$$
\begin{bmatrix}
a & b \\
c & d
\end{bmatrix}\cdot f(x) := (-cx + a)^{n-1}f\Big(\frac{dx -b}{-cx + a}\Big).
$$
That is, if $f$ is nonzero and $r_1, \dots, r_{n-1}\in\mathbb{P}^1(\mathbb{C})$ are the zeros of $f$ counted with multiplicity, and $\alpha\in\SL_2(\mathbb{C})$, then $\alpha\cdot r_1, \dots, \alpha\cdot r_{n-1}$ are the zeros of $\alpha\cdot f$. Via \eqref{polynomial_identification} (with $\mathbb{R}$ replaced by $\mathbb{C}$), we obtain an action of $\SL_2(\mathbb{C})$ on $\Gr_{k,n}(\mathbb{C})$. In particular, for $\zeros{X}$ as defined in \cref{defn_plane_span}, we have $\alpha\cdot\zeros{X} = \zeros{\alpha\cdot X}$.
\end{defn}

The compatibility of the $\SL_2(\mathbb{C})$-action extends to Wronskians:
\begin{lem}[{Purbhoo \cite[Proposition 2.1]{purbhoo10}}]\label{wronskian_SL}
For $V\in\Gr_{k,n}(\mathbb{C})$ and $\alpha\in\SL_2(\mathbb{C})$, we have $\Wr(\alpha\cdot V) = \alpha\cdot\Wr(V)$.
\end{lem}

We also consider the operator on $\mathbb{C}^n$ which reverses the order of the ground set.
\begin{defn}\label{defn_rev}
Given $n\in\mathbb{N}$, define the linear involution $\rev$ on $\mathbb{C}^n$ by
$$
\rev(a_1, \dots, a_n) := (a_n, \dots, a_1) \quad \text{ for all } a\in\mathbb{C}^n.
$$
Equivalently, $\rev$ acts on $\mathbb{C}[x]_{\le n-1}$ by
$$
\rev(f(x)) := x^{n-1}f(\textstyle\frac{1}{x}) \quad \text{ for all } f\in\mathbb{C}[x]_{\le n-1}.
$$
This induces an action $\rev$ on $\Gr_{k,n}(\mathbb{C})$.
\end{defn}

\begin{lem}\label{wronskian_rev}
For $V\in\Gr_{k,n}(\mathbb{C})$, we have $\rev(\Wr(V)) = \Wr(\rev(V))$.
\end{lem}

\begin{proof}
This follows from \cref{wronskian_SL}, since $\rev(f(x)) =  \scalebox{0.8}{$\begin{bmatrix}0 & 1 \\ -1 & 0\end{bmatrix}$}\cdot f(-x)$. Alternatively, we can apply \cref{wronskian_pluckers}.
\end{proof}

\subsection{Shift operator}\label{sec_shift}
We introduce the shift operator $x \mapsto x + t$ on polynomials, and discuss its connection to total positivity. Some of the results in this subsection have appeared before in the literature (see e.g.\ \cite[Section 4.2]{fallat01}), but for completeness, we give proofs.
\begin{defn}\label{defn_shift}
Given $n\in\mathbb{N}$ and $t\in\mathbb{C}$, define the {\itshape shift operator} $\shift{t}$ on $\mathbb{C}[x]_{\le n-1}$ by
$$
\shift{t}(f(x)) := f(x+t) \quad \text{ for all } f\in\mathbb{C}[x]_{\le n-1}.
$$
That is, $\shift{t}$ acts as the element $\scalebox{0.8}{$\begin{bmatrix}1 & -t \\ 0 & 1\end{bmatrix}$} \in \SL_2(\mathbb{C})$, in the sense of \cref{defn_SL_action}. We also let $D := \frac{d}{dx}$ denote the differentiation operator on $\mathbb{C}[x]_{\le n-1}$.
\end{defn}

Under the identification \eqref{polynomial_identification}, we may regard linear operators on $\mathbb{C}[x]_{\le n-1}$ as $n\times n$ matrices. The operator $D$ is given by the matrix whose nonzero entries are precisely $D_{i,i+1} = i$ for $1 \le i \le n-1$. For example, when $n=3$, we have $D = \scalebox{0.8}{$\begin{bmatrix}0 & 1 & 0 \\ 0 & 0 & 2 \\ 0 & 0 & 0\end{bmatrix}$}$.
\begin{lem}\label{shift_properties}
Let $n\in\mathbb{N}$ and $t\in\mathbb{C}$. Then the following three operators on $\mathbb{C}[x]_{\le n-1}$ are equal:
\begin{enumerate}[label=(\roman*), leftmargin=*, itemsep=2pt]
\item\label{shift_properties_shift} the shift operator $\shift{t}$;
\item\label{shift_properties_derivative} the exponential $\exp(tD)$; and
\item\label{shift_properties_binomial} the operator represented by the matrix $(\binom{j-1}{i-1}t^{j-i})_{1 \le i,j \le n}$.
\end{enumerate}

\end{lem}

For example, when $n=3$, we have $\shift{t} = \exp(tD) = \scalebox{0.8}{$\begin{bmatrix}1 & t & t^2 \\ 0 & 1 & 2t \\ 0 & 0 & 1\end{bmatrix}$}$.

\begin{proof}
\ref{shift_properties_shift} $=$ \ref{shift_properties_derivative}: It suffices to show that $\shift{t}(x^{j-1}) = \exp(tD)(x^{j-1})$ for $1 \le j \le n$. We have
$$
\exp(tD)(x^{j-1}) = \sum_{i\ge 0}\frac{(tD)^i}{i!}x^{j-1} = \sum_{i\ge 0}\binom{j-1}{i}t^ix^{j-i-1} = (x+t)^{j-1} = \shift{t}(x^{j-1}).
$$

\ref{shift_properties_shift} $=$ \ref{shift_properties_binomial}: The $(i,j)$-entry of the matrix representing $\shift{t}$ is the coefficient of $x^{i-1}$ in $\shift{t}(x^{j-1})$, which equals $\binom{j-1}{i-1}t^{j-i}$.
\end{proof}

\begin{rmk}\label{motivation_remark}
We observe that for $V\in\Gr_{k,n}(\mathbb{C})$, we have $\Wr(V) = \Delta_{[k]}(\shift{x}(V))$. Using the Cauchy--Binet identity \eqref{cauchy-binet} and the expression for $\shift{x}$ from \cref{shift_properties}\ref{shift_properties_binomial}, this leads to the formula for $\Wr(V)$ in \cref{wronskian_pluckers}.
\end{rmk}

We recall from \cref{sec_notation} the notion of total positivity for triangular matrices.
\begin{lem}\label{shift_positivity}
Let $0 \le k \le n$ and $I,J\in\binom{[n]}{k}$.
\begin{enumerate}[label=(\roman*), leftmargin=*, itemsep=2pt]
\item\label{shift_positivity_zero} If $I \nleq J$, then $\det((\shift{t})_{I,J}) = 0$ for all $t\in\mathbb{C}$.
\item\label{shift_positivity_positive} If $I \le J$, then there exists $c > 0$ such that $\det((\shift{t})_{I,J}) = ct^{\sumof{J} - \sumof{I}}$ for all $t\in\mathbb{C}$.
\end{enumerate}
In particular, $\shift{t}$ is a totally positive upper-triangular $n\times n$ matrix for all $t > 0$.
\end{lem}

We note that in \cref{shift_positivity}\ref{shift_positivity_positive}, when $I = [k]$ we have $c = \vand{J}$ (cf.\ \cref{motivation_remark}).

\begin{proof}
This follows from \cref{shift_properties} along with known results about totally positive upper-triangular matrices; we give two proofs. Firstly, a result of Lusztig \ \cite[Proposition 5.9(a)]{lusztig94} states that if $A$ is any $n\times n$ matrix such that $A_{i,i+1} > 0$ for $1 \le i \le n-1$ and all other entries of $A$ are zero, then $\exp(tA)$ is a totally positive upper-triangular matrix for all $t > 0$. Also, note that $\det(\exp(tA)_{I,J})$ is a monomial in $t$ of degree $\sumof{J} - \sumof{I}$. It now suffices to observe that the differentiation operator $D$ is of the form $A$, and $\shift{t} = \exp(tD)$.

Alternatively, let $B$ denote the $n\times n$ matrix with entries $B_{i,j} := \binom{j-1}{i-1}$ for $1 \le i,j \le n$, so that $\det((\shift{t})_{I,J}) = \det(B_{I,J})t^{\sumof{J} - \sumof{I}}$ by \cref{shift_properties}. We must show that $B$ is a totally positive upper-triangular matrix. Observe that $B_{i,j}$ is the number of paths from source $i$ to sink $j$ in the network \vspace*{3pt}
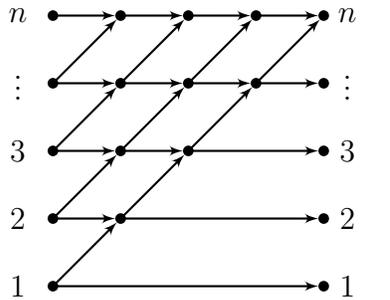
$$
\quad\begin{tikzpicture}[baseline=(current bounding box.center), scale=0.75]
\tikzstyle{out1}=[inner sep=0,minimum size=1.2mm,circle,draw=black,fill=black,semithick]
\pgfmathsetmacro{\s}{1.20};
\pgfmathsetmacro{\eps}{0.42};
\pgfmathsetmacro{\radius}{0.84};
\foreach \y in {2,...,5}{
\foreach \x in {2,...,\y}{
\node(v\x\y)[out1]at(\x*\s,\y*\s){};
\path[-latex',thick](\x*\s-\s,\y*\s)edge(v\x\y);}}
\foreach \y in {1,...,5}{
\node(source\y)[out1]at(1*\s,\y*\s){};}
\foreach \y in {1,...,4}{
\node(sink\y)[out1]at(5*\s,\y*\s){};
\path[-latex',thick](\y*\s,\y*\s)edge(sink\y);}
\foreach \y in {1,...,3}{
\node[inner sep=0]at(1-\eps,\y*\s){$\y$};
\node[inner sep=0]at(5*\s+\eps,\y*\s){$\y$};}
\node[inner sep=0]at(1-\eps,4*\s+0.06){$\vdots$};
\node[inner sep=0]at(5*\s+\eps,4*\s+0.06){$\vdots$};
\node[inner sep=0]at(1-\eps,5*\s){$n$};
\node[inner sep=0]at(5*\s+\eps,5*\s){$n$};
\foreach \y in {2,...,5}{
\foreach \x in {2,...,\y}{
\path[-latex',thick](\x*\s-\s,\y*\s-\s)edge(v\x\y);}}
\end{tikzpicture}\quad.\vspace*{3pt}
$$
Therefore by an involution principle due to Karlin and McGregor, Lindstr\"{o}m, and Gessel and Viennot (see e.g.\ \cite[Lemma 1]{fomin_zelevinsky00a}), $\det(B_{I,J})$ is the number of families of $k$ pairwise nonintersecting paths from the sources $I$ to the sinks $J$, which we see is positive for $I \le J$.
\end{proof}

We remark that our second proof of \cref{shift_positivity} does not use the fact that $\shift{t} = \exp(tD)$. However, we find that viewing $\shift{t}$ as $\exp(tD)$ clarifies why $\shift{t}$ (and hence the Wronskian) is compatible with total positivity.

\begin{cor}\label{shift_totally_positive}
Let $n\in\mathbb{N}$, and let $s,t > 0$. Then $\rev \circ \shift{s} \circ \rev \circ \shift{t}$ is a totally positive $n\times n$ matrix.
\end{cor}

\begin{proof}
By \cref{shift_positivity}, $\rev \circ \shift{s} \circ \rev$ (which is obtained by rotating the matrix $\shift{s}$ by $180^\circ$) is a totally positive lower-triangular matrix, and $\shift{t}$ is a totally positive upper-triangular matrix. Therefore by the Cauchy--Binet identity \eqref{cauchy-binet}, their product is a totally positive matrix.
\end{proof}

\section{Proof of \texorpdfstring{\cref{flag_wronskian}}{Theorem \ref{flag_wronskian}}}\label{sec_wronskians}

\noindent In this section we prove \cref{flag_wronskian}. We begin by collecting some preliminary results.
\begin{lem}\label{positive_coefficients}
Let $V\in\Gr_{k,n}(\mathbb{R})$.
\begin{enumerate}[label=(\roman*), leftmargin=*, itemsep=2pt]
\item\label{positive_coefficients_nonnegative} If $V$ is totally nonnegative, then up to rescaling, all coefficients of $\Wr(V)$ are nonnegative.
\item\label{positive_coefficients_positive} If $V$ is totally positive, then up to rescaling, the coefficient of $x^i$ in $\Wr(V)$ is positive for all $0 \le i \le k(n-k)$.
\end{enumerate}

\end{lem}

\begin{proof}
This follows from \cref{wronskian_pluckers}; cf.\ \cite[Theorem 5.5]{schechtman_varchenko20}.
\end{proof}

We reformulate \cref{connected_component_plucker} in terms of Wronskians:
\begin{lem}\label{connected_component_wronskian}
The totally positive flag variety $\Fl_n^{>0}$ is a connected component inside $\Fl_n(\mathbb{R})$ of the subset of $V = (V_1, \dots, V_{n-1})$ such that
\begin{gather*}
\Wr(V_k) \text{ is nonzero at } 0 \text{ and } \infty \quad \text{ for } 1 \le k \le n-1.
\end{gather*}

\end{lem}

\begin{proof}

Observe that by \cref{wronskian_pluckers}, for $W\in\Gr_{k,n}(\mathbb{R})$ we have
$$
\Delta_{[k]}(W) = 0 \hspace*{2pt}\Leftrightarrow\hspace*{2pt} \Wr(W) \text{ is zero at } 0 \quad \text{ and } \quad \Delta_{[n]\setminus [n-k]}(W) = 0 \hspace*{2pt}\Leftrightarrow\hspace*{2pt} \Wr(W) \text{ is zero at } \infty.
$$
Therefore the result is equivalent to \cref{connected_component_plucker}.
\end{proof}

We will also use the following consequence of \cref{shift_positivity}:
\begin{lem}\label{sufficiently_shifted_positive}
Let $V\in\Gr_{k,n}(\mathbb{R})$ be such that $\Wr(V)$ is nonzero at $\infty$. Then $\shift{t}(V)$ is totally positive for all $t > 0$ sufficiently large.
\end{lem}

\begin{proof}
Since $\Wr(V)$ is nonzero at $\infty$, \cref{wronskian_pluckers} implies that $\Delta_{[n]\setminus [n-k]}(V) \neq 0$. We fix a scaling of the Pl\"{u}cker coordinates so that $\Delta_{[n]\setminus [n-k]}(V) = 1$. Now let $I\in\binom{[n]}{k}$. By \cref{shift_positivity}\ref{shift_positivity_positive}, there exists $c > 0$ such that $\det((\shift{t})_{I,[n]\setminus [n-k]}) = ct^{\sumof [n]\setminus [n-k] - \sumof I}$. By the Cauchy--Binet identity \eqref{cauchy-binet} and \cref{shift_positivity}, we have
$$
\Delta_I(\shift{t}(V)) = \sum_{J\in\binom{[n]}{k}}\det((\shift{t})_{I,J})\Delta_J(V) = ct^{\sumof [n]\setminus [n-k] - \sumof I} + \text{lower order terms}
$$
as $t \to \infty$. Therefore $\Delta_I(\shift{t}(V)) > 0$ for all $t > 0$ sufficiently large.
\end{proof}

We now proceed to the proof of \cref{flag_wronskian}.
\begin{proof}[Proof of \cref{flag_wronskian}\ref{flag_wronskian_positive}]
Let $\psubset\subseteq\Fl_n(\mathbb{R})$ denote the subset of all complete flags $V = (V_1, \dots, V_{n-1})$ such that $\Wr(V_k)$ is nonzero on $[0,\infty]$ for $1 \le k \le n-1$. We must show that $\psubset = \Fl_n^{>0}$. First observe that $\Fl_n^{>0} \subseteq \psubset$, by \cref{positive_coefficients}\ref{positive_coefficients_positive}. Then by \cref{connected_component_wronskian}, $\Fl_n^{>0}$ is a connected component of $\psubset$. Therefore it suffices to show that given $V = (V_1, \dots, V_{n-1})\in\psubset$, there exists a path in $\psubset$ connecting $V$ to $\Fl_n^{>0}$. We claim that $t \mapsto \shift{t}(V)$ for $t\in [0,t_1]$ is such a path, for $t_1 > 0$ sufficiently large. Indeed, by \cref{wronskian_SL}, we have $\shift{t}(\psubset) \subseteq \psubset$ for all $t \ge 0$, so the path lies in $\psubset$. Also, by \cref{sufficiently_shifted_positive}, $\shift{t}(V)$ is totally positive for all $t > 0$ sufficiently large.
\end{proof}

\begin{proof}[Proof of \cref{flag_wronskian}\ref{flag_wronskian_nonnegative}]
The forward direction follows from \cref{positive_coefficients}\ref{positive_coefficients_nonnegative}. Conversely, let $V\in\Fl_n(\mathbb{R})$ such that $\Wr(V_k)$ is nonzero on $(0,\infty)$, for $1 \le k \le n-1$. We must show that $V$ is totally nonnegative. For $t > 0$, define $\alpha(t) := \scalebox{0.8}{$\begin{bmatrix}1 & -t \\ -t & 1+t^2\end{bmatrix}$}\in\SL_2(\mathbb{R})$. Observe that $\alpha(t)$ sends the interval $(0,\infty)\subseteq\mathbb{P}^1(\mathbb{R})$ to the open interval between $-\frac{t}{1+t^2}$ and $-\frac{1}{t}$, which contains $[0,\infty]$. Hence by \cref{wronskian_SL}, $\Wr(\alpha(t)\cdot V_k)$ is nonzero on $[0, \infty]$ for $1 \le k \le n-1$. By \cref{flag_wronskian}\ref{flag_wronskian_positive}, we get $\alpha(t)\cdot V\in\Fl_n^{>0}$ for all $t > 0$. Taking $t \to 0$ gives $V\in\Fl_n^{\ge 0}$.
\end{proof}

We remark that the forward directions of both parts of \cref{flag_wronskian} hold for any partial flag variety, which follows from \cref{wronskian_pluckers}. However, the reverse directions do not hold for an arbitrary partial flag variety, as the following example shows. \cref{positivity_conjecture} indicates a way to repair the reverse direction in the case of $\Gr_{k,n}(\mathbb{R})$, by imposing a stronger assumption on the Wronskian.
\begin{eg}\label{eg_partial_flag}
Let $V = (V_1, V_2)$ be the partial flag in $\mathbb{R}^4$ represented by the matrix
$$
A := \begin{bmatrix}
1 & 0 \\
1 & 2 \\
1 & 1 \\
1 & 3
\end{bmatrix}.
$$
That is, $V_k$ (for $k=1,2$) is the $k$-dimensional subspace spanned by the first $k$ columns of $A$. Then
$$
\Wr(V_1) = 1 + x + x^2 + x^3 \quad \text{ and } \quad \Wr(V_2) = 1 + x + 4x^2 + x^3 + x^4. 
$$
We see that $\Wr(V_1)$ and $\Wr(V_2)$ are nonzero on $[0,\infty]$; in fact, their coefficients are all positive. However, $V$ is not totally nonnegative, since $\Delta_{12}(V)$ and $\Delta_{23}(V)$ have opposite signs.
\end{eg}

\section{Chebyshev systems and disconjugacy}\label{sec_disconjugate}

\subsection{Background}\label{sec_background_disconjugate}
We recall some results about Chebyshev systems and disconjugate differential equations. We refer to \cite{polya22, karlin_studden66, coppel71, krein_nudel'man73, zalik96} for further details.
\begin{defn}\label{defn_detstar}
Let $x_1, \dots, x_k\in\mathbb{R}$. For $1 \le i \le k$, let $p_i$ be the number of elements among $x_1, \dots, x_{i-1}$ which are equal to $x_i$. For real-valued functions $f_1, \dots, f_k$ which are defined and sufficiently differentiable at $x_1, \dots, x_k$, we let
\begin{align}\label{detstar_equation}
\detstar(f_1, \dots, f_k; x_1, \dots, x_k) := \det(f_j^{(p_i)}(x_i))_{1 \le i,j \le k},
\end{align}
where $f_j^{(p_i)}$ denotes the $p_i$th derivative of $f_j$. In particular, if $x_1, \dots, x_k$ are distinct, then \eqref{detstar_equation} equals $\det(f_j(x_i))_{1 \le i,j \le k}$, while if $x_1 = \cdots = x_k = x$, then $\eqref{detstar_equation}$ equals the Wronskian of $f_1, \dots, f_k$ at $x$.
\end{defn}

We can equivalently express $\detstar$ as a limit, as follows:
\begin{lem}[{\cite[(2.8)]{karlin_studden66}; cf.\ \cite[Problems V.95 and V.96]{polya_szego25}}]\label{detstar_limit}
Let $x_1, \dots, x_k\in\mathbb{R}$. For $1 \le i \le k$, let $p_i$ be the number of elements among $x_1, \dots, x_{i-1}$ which are equal to $x_i$. Then
$$
\detstar(f_1, \dots, f_k; x_1, \dots, x_k) = \Big(\prod_{i=1}^kp_i!\Big)\lim\Big(\prod_{\substack{1 \le i < j \le k, \\ x_i = x_j}}\frac{1}{y_j - y_i}\Big)\det(f_j(y_i))_{1 \le i,j \le k},
$$
where the limit above is taken over all distinct $y_1, \dots, y_k\in\mathbb{R}$ as $y_i \to x_i$ for $1 \le i \le k$.
\end{lem}

We also have the following characterization of when $\detstar$ is zero:
\begin{lem}\label{detstar_lemma}
Let $x_1, \dots, x_k\in\mathbb{R}$, and let $f_1, \dots, f_k$ be real-valued functions which are defined and sufficiently differentiable at $x_1, \dots, x_k$. Then $\detstar(f_1, \dots, f_k; x_1, \dots, x_k) = 0$ if and only if there exist $c_1, \dots, c_k\in\mathbb{R}$ not all zero such that $c_1f_1 + \cdots + c_kf_k$ has zeros at $x_1, \dots, x_k$ (counted with multiplicity).
\end{lem}

\begin{proof}
This follows from the fact that a vector $c\in\mathbb{R}^k$ lies in the kernel of the matrix on the right-hand side of \eqref{detstar_equation} if and only if $c_1f_1 + \cdots + c_kf_k$ has zeros at $x_1, \dots, x_k$.
\end{proof}

\begin{defn}\label{defn_disconjugate}
Given an interval $I\subseteq\mathbb{R}$ (which may be open, closed, or half-open), let $C^k(I)$ denote the vector space over $\mathbb{R}$ of $k$-times continuously differentiable functions from $I$ to $\mathbb{R}$. Let $f_1, \dots, f_k\in C^k(I)$ be linearly independent, and let $V\subseteq C^k(I)$ denote the span of $f_1, \dots, f_k$.
\begin{itemize}[leftmargin=18pt, itemsep=2pt]
\item We say that $V$ is a {\itshape Chebyshev space (on $I$)} if every nonzero $f\in V$ has at most $k-1$ zeros in $I$, not counted with multiplicity. Equivalently, by \cref{detstar_lemma},
$$
\det(f_j(x_i))_{1 \le i,j \le k} \neq 0 \quad \text{ for all } x_1 < \cdots < x_k \text{ in } I.
$$
\item We say that $V$ is {\itshape disconjugate (on $I$)} if every nonzero $f\in V$ has at most $k-1$ zeros in $I$, counted with multiplicity. Equivalently, by \cref{detstar_lemma},
$$
\detstar(f_1, \dots, f_k; x_1, \dots, x_k) \neq 0 \quad \text{ for all } x_1 \le \cdots \le x_k \text{ in } I.
$$

\item We say that $(f_1, \dots, f_k)$ is a {\itshape Markov system (on $I$)}, or a {\itshape Markov basis of $V$ (on $I$)}, if
$$
\Wr(f_1, \dots, f_i) \text{ is nonzero on } I \quad \text{ for } 1 \le i \le k.
$$
We say that $V$ is a {\itshape Markov space} if it has a Markov basis.
\end{itemize}

\end{defn}

\begin{eg}
Let $V$ be the $2$-dimensional space spanned by $1$ and $x^3$, and let $I := (-1,1)$. Then $V$ is a Chebyshev space on $I$, but it is not disconjugate or a Markov space on $I$.
\end{eg}

\begin{rmk}\label{rmk_markov}
The three kinds of spaces introduced in \cref{defn_disconjugate} are closely related. By definition, a disconjugate space is a Chebyshev space. Also, a Markov space is disconjugate (see \cref{markov_disconjugate_chebyshev}); this follows from a clever application of the mean value theorem. Conversely, a Chebyshev space with a nonvanishing Wronskian on an open interval is disconjugate (see \cref{chebyshev_disconjugate_markov_chebyshev}), and a disconjugate space on an open or compact interval is a Markov space (see \cref{chebyshev_disconjugate_markov_disconjugate}).

We point out that the terms Chebyshev space, disconjugate space, and Markov space do not have standardized meanings in the literature. Some authors use different names, or define the properties for a system of functions $(f_1, \dots, f_k)$, rather than for its linear span, or require that the determinants appearing in each case are positive, rather than nonzero. (By continuity and connectedness, using \cref{detstar_limit}, if the determinants are nonzero then they take a constant sign, so the additional positivity requirement can be satisfied by negating some of the $f_i$'s.) \cref{comparison_table} summarizes the terminology used in \cite{karlin_studden66, coppel71, krein_nudel'man73, zielke79}. Each cell of the table records the name used in the corresponding reference (above the dashed line), along with the modification to the definition as compared to \cref{defn_disconjugate} (below the dashed line). Also, $C$, $E$, $M$, and $T$ stand for {\itshape complete}, {\itshape extended}, {\itshape Markov}, and {\itshape Chebyshev}, respectively.
\end{rmk}
\vspace*{0pt}\begin{table}[ht]
\begin{center}
\footnotesize\setlength{\tabcolsep}{6pt}
\begin{tabular}{| >{\vspace*{4pt}\centering\arraybackslash}m{1.32in}<{\vspace*{2pt}} | >{\vspace*{4pt}\centering\arraybackslash}m{1.32in}<{\vspace*{2pt}} | >{\vspace*{4pt}\centering\arraybackslash}m{1.32in}<{\vspace*{2pt}} | >{\vspace*{4pt}\centering\arraybackslash}m{1.32in}<{\vspace*{2pt}} |}
\cline{2-4}
\multicolumn{1}{l|}{} & Chebyshev space & disconjugate space & Markov system \\ \hline
\multirow{2}{1.32in}{\rule[10pt]{0pt}{0pt}Karlin and Studden \cite{karlin_studden66}} & $T$-system & $ET$-system & $ECT$-system \\ \cdashline{2-4}
 & $\det > 0$ & $\detstar > 0$ & $\Wr > 0$ \\ \hline
\multirow{2}{1.32in}{\vspace*{-5pt}Coppel \cite{coppel71}} & \multirow{2}{*}{\vspace*{-5pt}N/A} & disconjugate space, \v{C}eby\v{s}ev system & Markov system \\ \cdashline{3-4}
 & & no modification & $\Wr > 0$ \\ \hline
\multirow{2}{1.32in}{\rule[10pt]{0pt}{0pt}Krein and Nudel'man \cite{krein_nudel'man73}} & $T$-system & $ET$-system & $EM$-system\footnotemark \\ \cdashline{2-4}
 & no modification & no modification & no modification \\ \hline
\multirow{2}{1.32in}{\vspace*{-5pt}Zielke \cite{zielke79}} & \parbox{1.32in}{\centering Haar space, \\ \v{C}eby\v{s}ev system} & disconjugate space, $T^*$-system & $M^*$-system \\ \cdashline{2-4}
 & no modification & no modification & no modification \\ \hline
\end{tabular}\vspace*{8pt}
\caption{Comparison of terminology in \cref{defn_disconjugate} with existing literature.}
\label{comparison_table}
\end{center}
\end{table}\vspace*{-18pt}

\begin{rmk}\label{disconjugacy_equivalent}
We note that what Eremenko \cite{eremenko15} calls a disconjugate space is what we call a Chebyshev space. This is the notion that appears in \cref{disconjugacy_conjecture}, i.e., zeros are not counted with multiplicity. However, by \cref{chebyshev_disconjugate_markov_chebyshev}, it is equivalent in \cref{disconjugacy_conjecture} to count zeros with multiplicity.
\end{rmk}

\begin{rmk}
While differential equations do not play a role in our proofs, we mention that the motivation of P\'{o}lya \cite{polya22}, Coppel \cite{coppel71}, and many other authors is to study the linear differential operator $\mathcal{L}$ from \eqref{differential_operator} with a given kernel $V\subseteq C^k(I)$. It turns out that the property of $V$ being disconjugate or having a Markov basis is closely related to properties of $\mathcal{L}$. For example, if $V$ has a Markov basis $(f_1, \dots, f_k)$, then $\mathcal{L}$ admits the factorization
$$
\mathcal{L} = \bigg(\frac{h_k}{h_{k-1}}\frac{d}{dx}\frac{h_{k-1}}{h_k}\bigg)\bigg(\frac{h_{k-1}}{h_{k-2}}\frac{d}{dx}\frac{h_{k-2}}{h_{k-1}}\bigg) \cdots \bigg(h_1\frac{d}{dx}\frac{1}{h_1}\bigg),
$$
where $h_i := \Wr(f_1, \dots, f_i)$ for $1 \le i \le k$. We refer to \cite{coppel71} for further details. This connection also explains why, in \cref{defn_disconjugate}, we assume that $f_1, \dots, f_k$ lie in $C^k(I)$, rather than merely $C^{k-1}(I)$.
\end{rmk}

We collect results relating the three kinds of spaces in \cref{defn_disconjugate}. In addition to the references cited, for \cref{markov_disconjugate_chebyshev}, also see \cite[Theorem II]{polya22}, \cite[Theorem XI.1.1]{karlin_studden66}, \cite[Proposition 3.5]{coppel71}, and \cite[Theorem II.5.1]{krein_nudel'man73}; for \cref{chebyshev_disconjugate_markov_chebyshev}, see \cite[Proposition 3.3]{coppel71}; and for \cref{chebyshev_disconjugate_markov_disconjugate}, see \cite[Theorem IV]{polya22} and \cite[Theorem 3.3]{coppel71}.
\footnotetext{The term {\itshape $EM$-system} does not appear in \cite{krein_nudel'man73}, but is consistent with the naming conventions therein.}
\begin{thm}[{Markov \cite[Section 1]{markoff04}}]\label{markov_disconjugate_chebyshev}
Let $I\subseteq\mathbb{R}$ be an interval, and let $V$ be a $k$-dimensional Markov subspace of $C^k(I)$. Then $V$ is disconjugate on $I$.
\end{thm}

\begin{thm}[{Hartman \cite[Section 1]{hartman58}}]\label{chebyshev_disconjugate_markov_chebyshev}
Let $I\subseteq\mathbb{R}$ be an open interval, and let $V$ be a $k$-dimensional Chebyshev subspace of $C^k(I)$. If $\Wr(V)$ is nonzero on $I$, then $V$ is disconjugate on $I$.
\end{thm}

\begin{thm}[{Zielke \cite[Theorem 23.3]{zielke79}; Hartman \cite[Proposition 3.1]{hartman69}}]\label{chebyshev_disconjugate_markov_disconjugate}
Let $I\subseteq\mathbb{R}$ be an open or compact interval, and let $V$ be a $k$-dimensional disconjugate subspace of $C^k(I)$. Then $V$ is a Markov space on $I$.
\end{thm}

We point out that \cref{chebyshev_disconjugate_markov_disconjugate} does not hold for half-open intervals. For example, consider the $2$-dimensional space $V$ spanned by $x$ and $x^2-1$, with $I := [-1,1)$ \cite[Example 10.1]{zielke79}.

\subsection{Proof of \texorpdfstring{\cref{conjectures_equivalent}\ref{conjectures_equivalent_disconjugacy}}{Theorem \ref{conjectures_equivalent}\ref{conjectures_equivalent_disconjugacy}}}\label{sec_disconjugate_proof}
In this subsection we prove \cref{conjectures_equivalent}\ref{conjectures_equivalent_disconjugacy}, along with some other results stated in \cref{sec_introduction}. Recall the bilinear pairing introduced in \cref{sec_pairing}.
\begin{prop}\label{positivity_conjecture_duality}
Let $0 \le k \le n$. Then for $V\in\Gr_{k,n}(\mathbb{R})$, \cref{positivity_conjecture}\ref{positivity_conjecture_nonnegative} holds for $V$ if and only if \cref{positivity_conjecture}\ref{positivity_conjecture_nonnegative} holds for $V^\perp$. In particular, \cref{positivity_conjecture}\ref{positivity_conjecture_nonnegative} holds for $\Gr_{k,n}(\mathbb{R})$ if and only if it holds for $\Gr_{n-k,n}(\mathbb{R})$. The statements above hold with ``\cref{positivity_conjecture}\ref{positivity_conjecture_nonnegative}'' replaced by ``\cref{positivity_conjecture}\ref{positivity_conjecture_positive}''.
\end{prop}

\begin{proof}
This follows from \cref{perpendicular_pluckers} and \cref{perpendicular_wronskian}.
\end{proof}

\begin{prop}\label{positivity_conjecture_equivalence}
Let $0 \le k \le n$.
\begin{enumerate}[label=(\roman*), leftmargin=*, itemsep=2pt]
\item\label{positivity_conjecture_equivalence_wronskian} \cref{positivity_conjecture}\ref{positivity_conjecture_nonnegative} holds for $\Gr_{k,n}(\mathbb{R})$ if and only if \cref{positivity_conjecture}\ref{positivity_conjecture_positive} holds for $\Gr_{k,n}(\mathbb{R})$.
\item\label{positivity_conjecture_equivalence_secant} \cref{positive_secant_conjecture}\ref{positive_secant_conjecture_nonnegative} holds for $\Gr_{n-k,n}(\mathbb{C})$ if and only if \cref{positive_secant_conjecture}\ref{positive_secant_conjecture_positive} holds for $\Gr_{n-k,n}(\mathbb{C})$.
\end{enumerate}

\end{prop}

\begin{proof}
\ref{positivity_conjecture_equivalence_wronskian} ($\Leftarrow$) Suppose that \cref{positivity_conjecture}\ref{positivity_conjecture_positive} holds for $\Gr_{k,n}(\mathbb{R})$, and let $V\in\Gr_{k,n}(\mathbb{R})$ such that all zeros of $\Wr(V)$ lie in the interval $[-\infty,0]$. We must show that $V$ is totally nonnegative. For $t > 0$, define $\alpha(t) := \scalebox{0.8}{$\begin{bmatrix}1 & -t \\ -t & 1+t^2\end{bmatrix}$}\in\SL_2(\mathbb{R})$. By \cref{wronskian_SL}, all zeros of $\Wr(\alpha(t)\cdot V)$ lie in $\alpha(t) \cdot [-\infty,0] = (-\frac{1}{t}, \frac{-t}{1+t^2})$, and in particular, they lie in $(-\infty, 0)$. Hence by \cref{positivity_conjecture}\ref{positivity_conjecture_positive}, $\alpha(t)\cdot V$ is totally positive. Taking $t \to 0$ gives the result.

($\Rightarrow$) Suppose that \cref{positivity_conjecture}\ref{positivity_conjecture_nonnegative} holds for $\Gr_{k,n}(\mathbb{R})$, and let $V\in\Gr_{k,n}(\mathbb{R})$ such that all zeros of $\Wr(V)$ lie in $(-\infty,0)$. We must show that $V$ is totally positive. For $t > 0$, define $V(t) := (\shift{-t} \circ \rev \circ \shift{-t} \circ \rev)(V)\in\Gr_{k,n}(\mathbb{R})$. By \cref{wronskian_SL} and \cref{wronskian_rev}, we may fix $t > 0$ sufficiently small so that all zeros of $\Wr(V(t))$ lie in $(-\infty,0)$. By \cref{positivity_conjecture}\ref{positivity_conjecture_nonnegative}, $V(t)$ is totally nonnegative. Observe that $V = (\rev \circ \shift{t} \circ \rev \circ \shift{t})(V(t))$, and $\rev \circ \shift{t} \circ \rev \circ \shift{t}$ is a totally positive $n\times n$ matrix by \cref{shift_totally_positive}. Therefore $V$ is totally positive by the Cauchy--Binet identity \eqref{cauchy-binet}.

\ref{positivity_conjecture_equivalence_secant} We adopt the assumptions of \cref{positive_secant_conjecture}, so that $W_l = \Rnc{X_l}$ for $1 \le l \le k(n-k)$. By \cref{dual_problem}, we can write the Schubert problem \eqref{schubert_problem} as
$$
U^\perp\cap\zeros{-X_l}\neq\{0\} \quad \text{ for } 1 \le l \le k(n-k).
$$
Recall that $\alpha\cdot\zeros{-X_l} = \zeros{\alpha\cdot(-X_l)}$ for all $\alpha\in\SL_2(\mathbb{C})$, and also that by \cref{perpendicular_pluckers}, the map $\cdot^\perp$ preserves total nonnegativity and total positivity. Therefore writing $V = U^\perp$, we can use a similar argument as in the proof of part \ref{positivity_conjecture_equivalence_wronskian} above.
\end{proof}

We now proceed to the proof of \cref{conjectures_equivalent}\ref{conjectures_equivalent_disconjugacy}. We will show that \cref{positivity_conjecture}\ref{positivity_conjecture_nonnegative} implies \cref{disconjugacy_conjecture} and that \cref{disconjugacy_conjecture} implies \cref{positivity_conjecture}\ref{positivity_conjecture_positive}, which is sufficient by \cref{positivity_conjecture_equivalence}\ref{positivity_conjecture_equivalence_wronskian}.
\begin{proof}[Proof that \cref{positivity_conjecture}\ref{positivity_conjecture_nonnegative} for $\Gr_{k,n}(\mathbb{R})$ implies \cref{disconjugacy_conjecture} for $\Gr_{k,n}(\mathbb{R})$]
Suppose that \cref{positivity_conjecture}\ref{positivity_conjecture_nonnegative} holds for $\Gr_{k,n}(\mathbb{R})$. Given $V\in\Gr_{k,n}(\mathbb{R})$ such that all zeros of $\Wr(V)$ are real, and an interval $I\subseteq\mathbb{R}$ on which $\Wr(V)$ is nonzero, we must show that $V$ is a Chebyshev space on $I$. By \cref{wronskian_SL}, after acting by an element of $\SL_2(\mathbb{R})$, we may assume that all zeros of $\Wr(V)$ lie in $[-\infty, 0]$, and $I\subseteq (0,\infty)$. By \cref{positivity_conjecture}\ref{positivity_conjecture_nonnegative}, $V$ is totally nonnegative.

We now give two ways to complete the proof. Firstly, by \cref{flag_extension}, there exists a totally nonnegative complete flag $W = (W_1, \dots, W_{n-1})\in\Fl_n^{\ge 0}$ with $W_k = V$. By \cref{flag_wronskian}\ref{flag_wronskian_nonnegative}, $W$ is a Markov system on $(0,\infty)$, so $W_k = V$ is a Markov space. By \cref{markov_disconjugate_chebyshev}, $V$ is a Chebyshev space.

Alternatively, let $f\in V$ be a nonzero polynomial. By \cref{gantmakher_krein}, the sequence of coefficients of $f$ changes sign at most $k-1$ times. Therefore by Descartes's rule of signs, $f$ has at most $k-1$ zeros in $(0, \infty)$.
\end{proof}

\begin{proof}[Proof that \cref{disconjugacy_conjecture} for $\Gr_{k,n}(\mathbb{R})$, $\Gr_{n-k,n}(\mathbb{R})$ implies \cref{positivity_conjecture}\ref{positivity_conjecture_positive} for $\Gr_{k,n}(\mathbb{R})$]
Suppose that \cref{disconjugacy_conjecture} holds for $\Gr_{k,n}(\mathbb{R})$ and $\Gr_{n-k,n}(\mathbb{R})$. Given $V\in\Gr_{k,n}(\mathbb{R})$ such that all zeros of $\Wr(V)$ lie in $(-\infty,0)$, we must show that $V$ is totally positive. Let $I\subseteq\mathbb{P}^1(\mathbb{R})$ be an open interval containing $[0,\infty]$ such that $\Wr(V)$ is nonzero on $I$. (The results we are about to apply are stated for intervals of $\mathbb{R}$, not of $\mathbb{P}^1(\mathbb{R})$, but they are still valid. One way to see this is to act by an element of $\SL_2(\mathbb{R})$ so that $I$ does not contain $\infty$.) By \cref{disconjugacy_conjecture} and \cref{chebyshev_disconjugate_markov_chebyshev}, $V$ is disconjugate on $I$. By \cref{chebyshev_disconjugate_markov_disconjugate}, $V$ has a Markov basis $(f_1, \dots, f_k)$ on $[0,\infty]$. By an analogous argument, using \cref{perpendicular_wronskian}, $V^\perp$ has a Markov basis $(g_1, \dots, g_{n-k})$ on $[0,\infty]$.

Define the complete flag $W = (W_1, \dots, W_{n-1})\in\Fl_n(\mathbb{R})$ by
$$
W_i := \spn(f_1, \dots, f_i) \text{ for } 1 \le i \le k, \quad W_i := \spn(g_1, \dots, g_{n-i})^\perp \text{ for } k \le i \le n-1.
$$
Then $\Wr(W_i)$ is nonzero on $[0,\infty]$ for $1 \le i \le n-1$, where for $i > k$ we again use \cref{perpendicular_wronskian}. By \cref{flag_wronskian}\ref{flag_wronskian_positive}, we have $W\in\Fl_n^{>0}$. Since $V = W_k$, we obtain $V\in\Gr_{k,n}^{>0}$.
\end{proof}

\subsection{Proof of \texorpdfstring{\cref{conjectures_equivalent}\ref{conjectures_equivalent_secant}}{Theorem \ref{conjectures_equivalent}\ref{conjectures_equivalent_secant}}}\label{sec_secant_proof}
In this subsection we prove \cref{conjectures_equivalent}\ref{conjectures_equivalent_secant}, using an argument of Eremenko \cite{eremenko15}. We will need the following generalization of the intermediate value theorem, due to Miranda \cite{miranda40}:
\begin{thm}[Poincar\'{e}--Miranda \cite{kulpa97}]\label{poincare-miranda}
Let $C := [r_1, s_1] \times \cdots \times [r_n, s_n] \subseteq\mathbb{R}^n$ be the product of $n$ compact intervals of $\mathbb{R}$. For $1 \le i \le n$, let $\phi_i : C \to\mathbb{R}$ be a continuous function such that
$$
\phi_i(a)\phi_i(b) \le 0 \quad \text{ for all $a,b\in\mathbb{R}^n$ such that $a_i = r_i$ and $b_i = s_i$}.
$$
Then there exists a point $a\in C$ such that $\phi_1(a) = \cdots = \phi_n(a) = 0$. 
\end{thm}

The following result is implicit in \cite{eremenko15} in the case that the multisets $X_l$ have no repeated elements (also see \cite[Lemma 2]{eremenko_gabrielov_shapiro_vainshtein06}); for completeness, we give a proof.
\begin{lem}[{Eremenko \cite[p.\ 341]{eremenko15}}]\label{topological_lemma} 
Let $0 \le k \le n$, and suppose that \cref{disconjugacy_conjecture} holds for $\Gr_{k,n}(\mathbb{R})$. Let $I_1, \dots, I_{k(n-k)}$ be pairwise disjoint intervals of $\mathbb{P}^1(\mathbb{R})$. For $1 \le l \le k(n-k)$, let $X_l$ be a multiset of $k$ points contained in $I_l$, and let $W_l := \Rnc{X_l}$. Consider the Schubert problem \eqref{schubert_problem} with $U = V^\perp$, for $V\in\Gr_{k,n}(\mathbb{C})$:
\begin{align}\label{topological_lemma_problem}
V^\perp\cap W_l\neq\{0\} \quad \text{ for } 1 \le l \le k(n-k).
\end{align}
Then there are $d_{k,n}$ distinct solutions $V\in\Gr_{k,n}(\mathbb{C})$ to \eqref{topological_lemma_problem}, where each such $V$ is real, and each interval $-I_l$ (for $1 \le l \le k(n-k)$) contains a zero of $\Wr(V)$.
\end{lem}

\begin{proof}
We assume that each interval $I_l$ (for $1 \le l \le k(n-k)$) is compact, by replacing $I_l$ with a compact subinterval, if necessary. By \cref{dual_problem}, we can write the Schubert problem \eqref{topological_lemma_problem} as
\begin{align}\label{topological_lemma_rewritten}
V\cap \zeros{-X_l}\neq\{0\} \quad \text{ for } 1 \le l \le k(n-k).
\end{align}
In particular, after acting by an element of $\SL_2(\mathbb{R})$, we may assume that every interval $I_l$ (for $1 \le l \le k(n-k)$) does not contain $\infty$, so that $-I_l = [r_l, s_l]$ for some $r_l \le s_l$ in $\mathbb{R}$. Let us relabel the intervals $I_1, \dots, I_{k(n-k)}$ so that $r_1 < \cdots < r_{k(n-k)}$. Set $C := (-I_1) \times \cdots \times (-I_{k(n-k)})\subseteq\mathbb{R}^{k(n-k)}$ and $S := \{a\in\mathbb{R}^{k(n-k)} : a_1 < \cdots < a_{k(n-k)}\}$, and note that $C\subseteq S$.

By \cref{reality}, there exists a continuous function
$$
S\to\Gr_{k,n}(\mathbb{R})^{d_{k,n}}, \quad a \mapsto (V_1(a), \dots, V_{d_{k,n}}(a)),
$$
where $V_1(a), \dots, V_{d_{k,n}}(a)$ are distinct and all have the same Wronskian, namely, the polynomial with zeros $a_1, \dots, a_{k(n-k)}$. Fix $1 \le m \le d_{k,n}$, and take a basis $f_1(a), \dots, f_k(a)\in\mathbb{C}[x]_{\le n-1}$ of $V_m(a)$, where $f_1, \dots, f_k : S\to\mathbb{C}[x]_{\le n-1}$ are continuous functions satisfying
\begin{align}\label{topological_lemma_wronskian}
\Wr(f_1(a), \dots, f_k(a)) = \prod_{l=1}^{k(n-k)}(x-a_l) \quad \text{ for all } a \in S.
\end{align}

For $1 \le l \le k(n-k)$, define the continuous function
$$
\phi_l : S \to \mathbb{R}, \quad a \mapsto \detstar(f_1(a), \dots, f_k(a) ; -X_l),
$$
where above, the elements of $-X_l$ are written in weakly increasing order. We claim that it suffices to show that for $1 \le l \le k(n-k)$, we have
\begin{align}\label{topological_lemma_claim}
\phi_l(a)\phi_l(b) \le 0 \quad \text{ for all $a,b\in\mathbb{R}^{k(n-k)}$ such that $a_l = r_l$ and $b_l = s_l$}.
\end{align}
Indeed, then \cref{poincare-miranda} implies that there exists a point $a\in C$ (depending on $m$) such that $\phi_1(a) = \cdots = \phi_{k(n-k)}(a) = 0$. By \cref{detstar_lemma}, $V_m(a)$ is a solution of \eqref{topological_lemma_rewritten}, and hence also of \eqref{topological_lemma_problem}. Also, the interval $-I_l$ (for $1 \le l \le k(n-k)$) contains $a_l$, which is a zero of $\Wr(V_m(a))$. Hence the elements $V_m(a)$ for $1 \le m \le d_{k,n}$ provide $d_{k,n}$ distinct solutions of \eqref{topological_lemma_problem} with the desired properties.

We now fix $1 \le l \le k(n-k)$, and verify that \eqref{topological_lemma_claim} holds. Take $q < r_l$ and $t > s_l$ so that the open interval $(q,t)$ is disjoint from the intervals $-I_1, \dots, -I_{k(n-k)}$, except for $-I_l$. Set
$$
\tilde{C} := (-I_1) \times \cdots \times (-I_{l-1}) \times (q,t) \times (-I_{l+1}) \times \cdots \times (-I_{k(n-k)})\subseteq\mathbb{R}^{k(n-k)}.
$$
By \cref{disconjugacy_conjecture} and \cref{disconjugacy_equivalent}, the space $V_m(a)$ is disconjugate on every interval disjoint from $\{a_1, \dots, a_{k(n-k)}\}$. By \cref{defn_disconjugate} and \cref{detstar_limit},
$$
\detstar(f_1(a), \dots, f_k(a) ; x_1, \dots, x_k)
$$
is nonzero of a constant sign for all $a\in\tilde{C}$ and $x_1 \le \cdots \le x_k$ in the interval $(a_l, t)$. Note that the expression above equals $\Wr(f_1(a), \dots, f_k(a))$ when $x_1 = \cdots = x_k = x$, which has sign $(-1)^{k(n-k)-l}$ by \eqref{topological_lemma_wronskian}. Writing each $a\in C$ with $a_l = r_l$ as a limit of points $\tilde{a}\in\tilde{C}$ with $\tilde{a}_l < r_l$, we find that
$$
(-1)^{k(n-k)-l}\phi_l(a) \ge 0 \quad \text{ for all $a\in\mathbb{R}^{k(n-k)}$ such that $a_l = r_l$}.
$$
Similarly, we obtain
$$
(-1)^{k(n-k)-l+1}\phi_l(b) \ge 0 \quad \text{ for all $b\in\mathbb{R}^{k(n-k)}$ such that $b_l = s_l$},
$$
and \eqref{topological_lemma_claim} follows.
\end{proof}

\begin{proof}[Proof of \cref{conjectures_equivalent}\ref{conjectures_equivalent_secant}]
We show that \cref{positivity_conjecture}\ref{positivity_conjecture_nonnegative} holds for $\Gr_{k,n}(\mathbb{R})$ if and only if \cref{positive_secant_conjecture}\ref{positive_secant_conjecture_nonnegative} holds for $\Gr_{n-k,n}(\mathbb{C})$. This suffices by \cref{positivity_conjecture_equivalence}.

($\Rightarrow$) Suppose that \cref{positivity_conjecture}\ref{positivity_conjecture_nonnegative} holds for $\Gr_{k,n}(\mathbb{R})$. By \cref{conjectures_equivalent}\ref{conjectures_equivalent_disconjugacy}, \cref{disconjugacy_conjecture} holds for $\Gr_{k,n}(\mathbb{R})$, so we may apply \cref{topological_lemma} to the Schubert problem in \cref{positive_secant_conjecture}\ref{positive_secant_conjecture_nonnegative}. We conclude that it has $d_{k,n}$ distinct solutions $U\in\Gr_{n-k,n}(\mathbb{R})$, where all zeros of $\Wr(U^\perp)$ lie in the interval $[-\infty,0]$. By \cref{positivity_conjecture}\ref{positivity_conjecture_nonnegative}, each $U^\perp$ is totally nonnegative, and so $U$ is totally nonnegative by \cref{perpendicular_pluckers}.

($\Leftarrow$) Suppose that \cref{positive_secant_conjecture}\ref{positive_secant_conjecture_nonnegative} holds for $\Gr_{n-k,n}(\mathbb{C})$, and consider the special case when each $X_l$ consists of a single element of multiplicity $k$. By \cref{dual_problem} and \cref{perpendicular_pluckers}, this special case can be stated equivalently as follows: if $V\in\Gr_{k,n}(\mathbb{C})$ such that all zeros of $\Wr(V)$ are distinct and lie in the interval $[-\infty,0]$, then $V$ is totally nonnegative. We obtain the same result when $\Wr(V)$ has repeated zeros by a limiting argument (see e.g.\ \cite[Section 1.3]{mukhin_tarasov_varchenko09a}).
\end{proof}

\subsection{Example: \texorpdfstring{$\Gr_{2,4}(\mathbb{R})$}{Gr(2,4)}}\label{eg_positivity_conjecture}
We verify \cref{positivity_conjecture}\ref{positivity_conjecture_positive} for $\Gr_{2,4}(\mathbb{R})$. Let $V\in\Gr_{2,4}(\mathbb{R})$ such that all zeros of $\Wr(V)$ lie in the interval $(-\infty, 0)$. We wish to show that $V$ is totally positive. We may write
$$
\Wr(V) = (1+r_1x)(1+r_2x)(1+r_3x)(1+r_4x) = 1 + e_1x + e_2x^2 + e_3x^3 + e_4x^4,
$$
where $r_1, r_2, r_3, r_4 > 0$, and $e_i$ (for $1 \le i \le 4$) is the $i$th elementary symmetric polynomial in $r_1, r_2, r_3, r_4$. Using \cref{eg_wronskian_pluckers} and the Pl\"{u}cker relation $\Delta_{13}\Delta_{24} = \Delta_{12}\Delta_{34} + \Delta_{14}\Delta_{23}$, we find that the Pl\"{u}cker coordinates of $V$ are
$$
\Delta_{12} = 1,\quad \Delta_{13} = \frac{e_1}{2},\quad \Delta_{14} = \frac{e_2 \pm \sqrt{\kappa}}{6},\quad \Delta_{23} = \frac{e_2 \mp \sqrt{\kappa}}{2},\quad \Delta_{24} = \frac{e_3}{2},\quad \Delta_{34} = e_4,
$$
where
$$
\kappa := e_2^2 - 3e_1e_3 + 12e_4 = \frac{(r_1-r_2)^2(r_3-r_4)^2 + (r_1-r_3)^2(r_2-r_4)^2 + (r_1-r_4)^2(r_2-r_3)^2}{2}.
$$
(The fact that $\kappa \ge 0$ is equivalent to \cref{reality} for $\Gr_{2,4}(\mathbb{C})$; cf.\ \cite[Example 2.2]{sottile00}.) Therefore $V$ is totally positive if and only if $e_2^2 > \kappa$. This is equivalent to $e_1e_3 > 4e_4$, which then follows by expanding both sides in $r_1, r_2, r_3, r_4$.

\bibliographystyle{alpha}
\bibliography{ref}

\end{document}